\documentclass[10pt,a4paper]{amsart}
\usepackage{amsmath, graphicx, subfigure, epsfig,amssymb}
\usepackage{cite,xcolor}
\usepackage[normalem]{ulem}

\numberwithin{equation}{section}

\newcommand{\e}{\epsilon}

\newcommand{\de}{\delta}
\newcommand{\br}{\mathbb{R}}

\newcommand{\pa}{\partial}
\newcommand{\bt}{\beta}
\newcommand{\al}{\alpha}
\newcommand{\la}{\lambda}
\newcommand{\om}{\omega}
\newcommand{\coi}{C_0^{\infty}}

\newcommand{\be}{\begin{equation}}
\newcommand{\ee}{\end{equation}}

\newcommand{\dd}{\text{d}}
\newcommand{\WF}{\text{WF}}
\newcommand{\frec}{f_\e^{\text{rec}}}
\newcommand{\freco}{f_\e^{\text{rec-1}}}
\newcommand{\frect}{f_\e^{\text{rec-2}}}
\newcommand{\frecth}{f_\e^{\text{rec-3}}}

\newcommand{\alst}{\alpha_\star}
\newcommand{\pst}{p_\star}
\newcommand{\rhost}{\rho_\star}
\newcommand{\kst}{k_\star}
\newcommand{\sst}{\s_\star}

\newcommand{\bma}{\begin{pmatrix}}
\newcommand{\ema}{\end{pmatrix}}

\newcommand{\s}{\mathcal S}

\newcommand{\R}{\mathcal R}

\newcommand{\mP}{\mathcal P}
\newcommand{\mps}{\mathcal P_*}

\newcommand{\CH}{\mathcal H}

\def\bt#1{\textcolor{black}{#1} }

\newtheorem{theorem}{Theorem}[section]
\newtheorem{lemma}[theorem]{Lemma}
\newtheorem{corollary}[theorem]{Corollary}

\theoremstyle{definition}
\newtheorem{assumptions}[theorem]{Assumptions}

\begin{document}

\title[Analysis of view aliasing]{Analysis of view aliasing for the generalized Radon transform in $\br^2$}
\author[A Katsevich]{Alexander Katsevich$^1$}
\thanks{$^1$This work was supported in part by NSF grant DMS-1906361. Department of Mathematics, University of Central Florida, Orlando, FL 32816 (Alexander.Katsevich@ucf.edu). }

\begin{abstract} In this paper we consider the generalized Radon transform $\mathcal R$ in the plane. Let $f$ be a piecewise smooth function, which has a jump across a smooth curve $\mathcal S$. We obtain a formula, which accurately describes view aliasing artifacts away from $\mathcal S$ when $f$ is reconstructed from the data $\mathcal R f$ discretized in the view direction. The formula is asymptotic, it is established in the limit as the sampling rate $\epsilon\to0$. The proposed approach does not require that $f$ be band-limited. Numerical experiments with the classical Radon transform and generalized Radon transform (which integrates over circles) demonstrate the accuracy of the formula. 
\end{abstract}
\maketitle

\section{Introduction}\label{sec_intro}

Resolution of image reconstruction from discrete data is one of the fundamental questions in imaging. The most direct approach to estimating resolution utilizes the notions of the point spread function (PSF) and modulation transfer function (MTF) \cite[Sections 12.2, 12.3]{eps08}. This and other similar approaches allow rigorous theoretical analysis of only the simplest settings, such as inversion of the classical Radon transform. For the most part, resolution of reconstruction in more difficult settings (e.g., inversion of the cone beam transform) is analyzed by heuristic arguments, numerically, or via measurements \cite{LAH2007, Grimmer2008, FFS2013}.

Sampling theory provides a \bt{related} approach to investigating resolution \cite{kr89, des93, nat93, nat95, pal95, cb97, fr00, far04,  rs07, izen12}. Consider, for example, the classical Radon transform in $\br^2$
\be
\hat f(\al,p)=\int_{\br^2} f(x)\de(\vec\al\cdot x-p)\dd x,\ \vec\al=(\cos\al,\sin\al).
\ee
The corresponding discrete data are
\be
\hat f(\al_k,p_j),\ \al_k=\bar\al+k\Delta\al,\ p_j=\bar p+j\Delta p,\ \al_k\in[0,2\pi),j\in\mathbb Z,
\ee
for some fixed $\bar\al$, $\bar p$ and $\Delta\al$, $\Delta p$. Assume that $f$ is essentially band-limited (in the classical sense). This means that, with high accuracy,  its Fourier transform $\tilde f(\xi)$ is supported in some ball $|\xi|\le B$. The sampling theory predicts the rates $\Delta\al$, $\Delta p$ with which $\hat f(\al,p)$ should be sampled, so that reconstruction of $f$ from discrete data does not contain aliasing artifacts. Since the essential band-limit $B$ is related to the size of the smallest detail in $f$, a typical prescription of the theory can be loosely formulated as follows: given the size of the smallest detail in $f$, the minimal sampling rates to avoid aliasing are $\Delta\al$, $\Delta p$. Alternatively, the theory determines the size of the smallest detail in $f$ that can be resolved given the rates $\Delta\al$, $\Delta p$.

A microlocal approach to sampling was developed recently \cite{stef20, Monard2021, Stef_angles_2022}. In this approach $f$ is assumed to be band-limited in the semiclassical sense (i.e., the semiclassical wavefront set $\WF_h(f)$ is compact). \bt{Alternatively, the assumption is that the data represent discrete values of the convolution $w*\R f$. Here $\R$ is the generalized Radon transform, and $w$ is a semiclassically bandlimited mollifier. The mollifier models the detector aperture function.} The goal is to accurately recover the semiclassical singularities of $f$ and avoid aliasing. If the sampling requirement is violated, the theory predicts the location and frequency of aliasing artifacts. 

%At a high level, the classical and semiclassical sampling theories are similar in two aspects. First, both assume that $f$ is band-limited. Second, in both of them the size of the smallest detail in $f$ determines the sampling rate of $\hat f$, and the sampling rate determines the smallest detail in $f$ that can be resolved. 

%This implies that the practical use of the \bt{existing results in the} sampling theory to precisely quantify the resolution of reconstruction is \bt{sometimes} not fully justified. Indeed, the vast majority of objects being scanned are nonsmooth and have sharp edges, so, strictly speaking, they are not band-limited. Therefore, specifying an approximate band-limit for the function $f$ that describes such an object is fairly arbitrary. Second, the notion of resolving a detail in $f$  is somewhat arbitrary too. \bt{When $f$ is not bandlimited}, there is no threshold such that the given feature is resolved if the sampling rate is smaller than the threshold, and the feature is not resolved if the sampling rate is greater than the threshold. There is always a smooth transition between resolving a feature and not resolving a feature. Finally, if a feature in $f$ is resolved according to the sampling theory, the theory does not specify what the reconstruction of the feature looks like. For these three reasons, in practice, the notions of a band-limit (both classical and semiclassical) and the size of the smallest detail provide a ballpark estimate of resolution rather than its precise value. 

In \cite{Katsevich2017a, kat19a, Katsevich2020b, Katsevich2020a, Katsevich2021a}, the author developed an \bt{alternative} analysis of resolution (we call it Local Resolution Analysis, or LRA).
%, which addresses the question of resolution head on. 
The main results in these papers are simple expressions describing the reconstruction from discrete \bt{values of $\R f$ or $w*\R f$} in a neighborhood of the singularities of $f$ in a variety of settings. We call these expressions the Discrete Transition Behavior (DTB). The DTB provides a direct, quantitative link between the sampling rate and resolution. In these papers such a link is established for a wide range of integral transforms, conormal distributions $f$, and reconstruction operators. In \cite{Katsevich2021b, Katsevich2022a} LRA was generalized to objects with rough boundaries in $\br^2$. \bt{Neither $f$ nor the mollifier $w$ (if applied) is required to be bandlimited.} 

Suppose $\Delta p=\e$ and $\Delta\al=\kappa\e$, where $\kappa>0$ is fixed. The DTB is an accurate approximation of the reconstruction in an $\e$-neighborhood of the singular support of $f$ in the limit as $\e\to0$. Therefore, the DTB provides much more than a single measure of resolution (e.g., the size of the smallest detail that can be resolved). Given the DTB function, the user may decide in a fully quantitative way what sampling rate is required to achieve a user-defined reconstruction quality. The notion of quality may include resolution (which can be described in any desired way) and/or any other requirement the user desires. Thus, the LRA answers precisely the question of the required sampling rate to guarantee the required resolution (understood broadly). 
%No assumption that $f$ be band-limited (either classically or semiclassically) is required. 

The only item missing from the LRA until now was analysis of aliasing. 
%In practice $f$ is almost never band-limited \bt{(and data never represent discrete values of the convolution of a bandlimited mollifier with the Radon transform of $f$)}, so aliasing artifacts arise regardless of the sampling rate. 
Some earlier results on the analysis of aliasing artifacts (more precisely, view aliasing artifacts) are in \cite{JS1980} and \cite[Section 12.3.2]{eps08}. They include an approximate formula for the artifacts far from a small, radially symmetric object. More recent results are in \cite{stef20, Monard2021, Stef_angles_2022}. These include the prediction of the location and frequency of the artifacts, qualitative analysis of the artifacts generated by various edges (e.g., flat, convex, and a corner), as well as their numerical illustrations.

In this paper we generalize the LRA to the analysis of view aliasing. We call it the Localized Aliasing Analysis, or LAA. Our main result is Theorem~\ref{main res}, where a precise, quantitative formula describing aliasing artifacts is stated. The formula is asymptotic, it is established in the limit as the sampling rate $\e\to0$ (which is the same assumption as in \cite{stef20, Monard2021, Stef_angles_2022}). Similarly to the LRA, the LAA is very flexible. In this paper we consider the generalized Radon transform in $\br^2$ and apply it to functions with jump discontinuities across smooth curves. Similarly to \cite{Katsevich2017a, kat19a, Katsevich2020b, Katsevich2020a, Katsevich2021a}, we believe that the LAA is generalizable, and that it is capable of predicting aliasing artifacts for a wide range of integral transforms, conormal distributions $f$, and reconstruction operators. 

\color{black}
To avoid confusion, we clarify the meaning of the terms “resolution” and “aliasing” used in this paper. For simplicity, we will use the example of a jump discontinuity across a smooth curve $\s$. Resolution at $x_0\in\s$ means the extent to which the boundary at the jump (i.e., $\s$) is blurred when the image is reconstructed in a neighborhood of $x_0$ from discrete data. This blurring is accurately described by the DTB function mentioned above. The derivation of the DTB function accounts for possible artifacts that may arise due to aliasing from the parts of $\s$ in a neighborhood of $x_0$. In other words, {\it LRA treats local aliasing as part of resolution analysis}. In this paper, the term {\it “aliasing” stands for rapidly oscillating artifacts away from $\s$ that are caused by aliasing from $\s$.}
\color{black}

The paper is organized as follows. In section~\ref{sec:a-prelims} we describe the set-up, formulate the assumptions, and state the main result -- Theorem~\ref{main res}. This theorem provides a simple formula that describes aliasing artifacts. \bt{We also discuss various quantities used in the formula, and state a corollary that describes what the formula looks like in the case of the classical Radon transform.} The proof of Theorem~\ref{main res} is in section~\ref{sec:proof}. Section~\ref{sec:PSI props} establishes a few useful properties of the function $\Psi$, in terms of which the artifacts are described. An algorithm for computing $\Psi$ numerically is in Section~\ref{sec:PSI numer}. Section~\ref{sec:numerics} contains numerical experiments with the classical and generalized Radon transforms. The latter integrates over circles. Details of implementation, which illustrate the use of the theorem, are provided. All experiments demonstrate a good match between reconstruction and prediction. Proofs of some lemmas are in appendix~\ref{sec:proofs of lemmas}.

\section{Preliminaries}\label{sec:a-prelims}

\subsection{Generalized Radon transform}\label{sec:gen fn}
Let $p=\mps(\al,x)$ be a defining function for the generalized Radon transform $\R$:
\be\label{grt def}
\hat f(\al,p)=\int_{\s_{\al,p}} W(\al,p;x) f(x)\dd A,\ \s_{\al,p}:=\{x\in\br^2:\mps(\al,x)=p\}, \al\in\Omega,p\in\br,
\ee
where $W\in C^\infty(\Omega\times\br\times U)$ is some (known) integration weight, $\dd A$ is the length element on the curve $\s_{\al,p}$, \bt{$U\subset\br^2$ is a small open set,} and $\Omega\subset\br$ is a small interval. 
%As is well-known, $\dd A=(\det \hat G(x;\al,p))^{1/2}\dd x$, where $\hat G$ is the Gram matrix
%\be\label{gram-1st}
%\hat G=\{G_{jk}\}_{j,k=1}^2,\ G_{jk}(x;\al,p)=\frac{\pa \Phi(x;\al,p)}{\pa x_j}\frac{\pa \Phi(x;\al,p)}{\pa x_k},\ 
%1\le j,k\le 2.
%\ee
Similarly to the classical Radon transform, we think about $\al$ as the polar angle, and $p$ - as the affine variable. However, since we consider the generalized Radon transform, these variables admit many alternative interpretations. \bt{See \cite{Kuchment2005, AmbQ2020} for more information and references about generalized Radon transforms, their properties and applications.}

Let $\s$ be a $C^\infty$ curve. Let $(\alst ,\pst)$ be a pair such that $\s_{\alst ,\pst}$ is tangent to $\s$ at some $y_0\in \s\bt{\cap U}$. \bt{To simplify notation, denote $\sst:=\s_{\alst ,\pst}$.} We will compute a reconstruction in a small neighborhood of some point $x_0\not\in\sst$.  
%Let $U$ be a small neighborhood of $\{x_0,y_0\}$. For example, $U$ can consist of two disconnected open sets: one - around $x_0$, and the other - around $y_0$.
Let $H(y)=0$ be an equation for $\s$ in a neighborhood of $y_0$. The function $H$ is smooth, and $\dd H(y)\not=0$, $y\in\s$. Multiplying $H$ by a constant if necessary, we can assume that $\mps$ satisfies the equations
\be\label{Phi eqs}
\mps(\alst,x_0)=\mps(\alst,y_0)=\pst,\ \dd_x\mps(\alst,y_0)=\dd H(y_0).
\ee
%for some $\la\in\br$.
%, where $\vec\Theta_0$ is a unit normal vector to $\s$ at $y_0$.

\begin{assumptions}[Properties of $\mps$]\label{ass:Phi}
$\hspace{1cm}$
\begin{enumerate}
\item $\mps\in C^\infty(\Omega\times \bt{U})$, and $\dd_x\mps(\al,x)\not=0$, $x=x_0,y_0$;
%\bt{
%\be
%\text{det}
%\bma \dd_x\mps(\al,x)\\ \dd_x(\pa_\al\mps(\al,x))\ema\not=0,\ x\in \sst\cap U,\al=\alst;
%\ee}
\item Equations \eqref{Phi eqs} hold;
\item $\pa_\al \mps(\alst,x_0)\not=\pa_\al \mps(\alst,y_0)$ \bt{(the Bolker condition);}
\item One has 
\be\label{p der cond}
M:=(\vec\Theta_0^\perp\cdot\pa_y)^2(\mps(\alst,y)-H(y))|_{y=y_0}>0,
\ee
where $\vec\Theta_0^\perp$ is a unit vector orthogonal to $\dd H(y_0)$; and
\item \bt{There exists $c>0$ such that $y_0\not\in\s_{\al,p}$ for any $\al\in\Omega$ and $|p-\pst|\ge c$.
}
%\item \bt{There is no other pair $(\al’,p’)\not=(\alst,\pst)$ such that $x_0\in\s_{\al’,p’}$, $\al’\in\Omega$, and $\s_{\al’,p’}$ is tangent to $\s$ at $y_0$.\red{use somewhere}
\end{enumerate}
\end{assumptions}

Assumption~\ref{ass:Phi}(4) \bt{is equivalent to the condition that the curvatures of $\s$ and $\sst$ at $y_0$ are not equal. In other words, the order of contact between $\s$ and $\sst$ is one (and not higher). For example, if one of the two curves is flat at $y_0$, then $M\not=0$ as long as the other one is not flat.} The requirement that $M$ be positive is not restrictive. If $M<0$, we can flip the $p$-axis and replace $H\to-H$, $\mps\to-\mps$ to make $M$ positive. The essential requirement is that $M\not=0$. 

The requirement $M>0$ means that $\s_{\alst,\pst+\de}$ intersects $\s$ at two points near $y_0$ when $\de>0$, and does not intersect $\s$ near $y_0$ -- if $\de<0$ (see Figure~\ref{fig:assns}).  \bt{In what follows we set
\be
\vec\Theta_0:=\pm\dd H(y_0),
\ee
and the sign ($+$ or $-$) is selected so that $\vec\Theta_0$ points towards the part of $\s_{\alst,\pst+\de}$, $0<\de\ll1$, located between its two intersection points with $\s$ (see Figure~\ref{fig:assns}).}

\bt{Shrinking, if necessary, $\Omega$ and $U$ further, we may assume that there is no other pair $(\al’,p’)\not=(\alst,\pst)$, $\al’\in\Omega$, such that $x_0\in\s_{\al’,p’}$, and $\s_{\al’,p’}$ is tangent to $\s$ at $y_0$.
}

\begin{figure}[h]
{\centerline{
{\epsfig{file={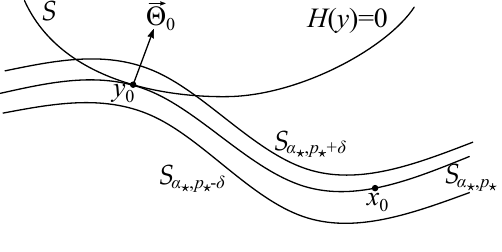}, width=8cm}}
}}
\caption{\bt{Illustration of the curves $\s$ and $\s_{\al,p}$.}}
\label{fig:assns}
\end{figure}

Let $\mP(\al)$, $\al\in\Omega$, be the function defined by the requirement that the curves $\s_{\al,\mP(\al)}$ be tangent to $\s$ in a neighborhood of $y_0$. \bt{Figure~\ref{fig:Pal} illustrates the function $\mP(\al)$ in the case of the classical Radon transform (left panel) and the generalized Radon transform that integrates over circles (right panel). The circles have arbitrary radii and centers on a given curve $z(\al)\in\Gamma$, $\al\in\Omega$. Consider the latter case. Suppose, for example, that $\s$ is a circle with radius $r$ and center $a$. Then, globally, there are two such functions: $\mP(\al)=|z(\al)-a|\pm r$. See also Section~\ref{sec:grt_exp} for more details about the circular Radon transform.}

The following simple lemma is proven in appendix~\ref{sec:Psm}.

\begin{figure}[h]
{\centerline{\hbox{
{\epsfig{file={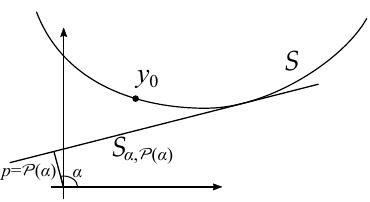}, width=6cm}}
{\epsfig{file={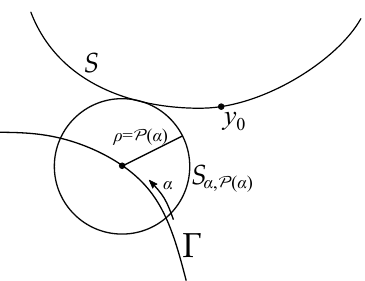}, width=6cm}}
}}}
\caption{\bt{Illustration of the function $\mP(\al)$. Left panel - the classical Radon transform that integrates over lines. Right panel - the generalized Radon transform that integrates over circles with centers on a given curve parametrized by $\al$ (denoted $\Gamma$ in the figure).}}
\label{fig:Pal}
\end{figure}

\begin{lemma}\label{lem:two Ps} For a sufficiently small neighborhood $\Omega\ni\alst$, one has
\be\label{two derivs}
\mP(\alst )=\pst,\ \mP(\al)\in C^\infty(\Omega),\
\bt{\pa_\al\mps(\alst ,y_0)=\mP’(\alst).}
\ee
\end{lemma}
\bt{
From assumptions~\ref{ass:Phi}(1, 3) and Lemma~\ref{lem:two Ps},  
\be\label{two derivs 1}
u_0:=\dd_x\mps(\alst ,x_0)\not=0,\ \mu_0:=\pa_\al(\mps(\alst ,x_0)-\mP(\alst ))\not=0.
\ee}
%\bt{As is seen from the proof of Lemma~\ref{lem:two Ps}, assumption~\ref{ass:Phi}(3) is key for establishing that $\mu_0\not=0$.}

\subsection{Remaining assumptions and main result}

Consider a function $f(x)$ on the plane, $x\in\br^2$. We suppose that 
\begin{assumptions}[Properties of $f$]\label{ass:f}
$\hspace{1cm}$
\begin{enumerate}
\item $\text{supp}(f)\subset U$, and $\text{diam}(U)$ is sufficiently small;
\item There exist open sets $D_\pm$ and functions $f_\pm\in C^\infty(\br^2)$ such that
\begin{equation}\label{f_def}\begin{split}
& f(x)\equiv f_-(x),\ x\in D_-,\  f(x)\equiv f_+(x),\ x\in D_+,\\
& D_-\cap D_+=\varnothing,\ D_-\cup D_+=U\setminus \s;
\end{split}
\end{equation}
and
\item $\s\cap U$ is a $C^\infty$ curve.
%\item[f3.]\label{ass:curv} There are finitely many points $x\in \s$ where the curvature of $\s$ equals zero, and these zeroes are of finite order.
\end{enumerate}
\end{assumptions}

\noindent
Thus, $\text{sing\,supp}(f)\subset\s$. In general, $f_-(x)\not=f_+(x)$, $x\in\s$, so $f$ may have a jump across $\s$. Note that whether $x_0\in U$ or not is irrelevant. Also, when $U$ shrinks towards $y_0$, $\s$ does not change. Thus, $\s\cap U$ is a small segment of $\s$ around $y_0$. With this understanding, in what follows we do not distinguish between $\s$ and $\s\cap U$. 

Similarly to \cite{Stef_angles_2022}, we consider semi-discrete data
\be\label{data_eps}
\hat f_\e(\al_k,p):=\int w_\e(p-s)\hat f(\al_k,s)\dd s,\  
\al_k:=k\Delta\al,\ p\in\br,\ w_\e(t):=\e^{-1}w(t/\e),
\ee
where $w$ is a mollifier (e.g., the detector aperture function), $\Delta\al=\kappa\e$, and $\kappa>0$ is fixed. \bt{It is reasonable to assume that the support of $w_\e$ is
% Also, we assume that the data has been filtered with respect to $p$ by convolving it with a smooth filter with support 
 of size $O(\e)$, because sampling rates along $\al$ and $p$ are usually of the same order of magnitude.}
%\textcolor{red}{Here and below, $\vec \al$ and $\al$ in the same equation are always related by $\vec\al=(\cos\al,\sin\al)$. The same applies to $\vec\Theta=(\cos\theta,\sin\theta)$ and $\theta$.}

\begin{assumptions}[Assumptions about the mollifier $w$]\label{ass:w} 
$\hspace{1cm}$
\begin{enumerate}
\item \bt{$w$ is compactly supported and $w’\in L^q(\br)$ for some $q>2$;} and 
\item $\int w(p)\dd p=1$.
\end{enumerate}
\end{assumptions}
Hence, the data \eqref{data_eps} represent the integrals of $f$ along thin strips around $\s_{\al_k,p}$, and their width ($=O(\e)$) is determined by $\e$ and the support of $w$. In the ideal case \bt{(not considered in this paper)}, where $w$ is the Dirac $\de$-function, the data represent the integrals of $f$ along $\s_{\al_k,p}$.

Reconstruction from the data \eqref{data_eps} is achieved by the formula
\be\label{recon-orig}
\frec(x)=-\frac{\Delta\al}{2\pi}\sum_{\al_k\in\Omega} \frac{\omega(\al_k,x)}\pi \int \frac{\pa_p\hat f_\e(\al_k,p)}{p-\mps(\al_k,x)}\dd p,\ x\in U’,
\ee
where $U’$ is a small neighborhood of $x_0$, and $\omega\in C^\infty(\Omega\times U’)$ is some weight function. This is a discretized (in $\al$) version of the classical FBP inversion formula \cite{nat3} adapted to the generalized Radon transform in $\br^2$ (e.g., as it was done in \cite{Beylkin1984, kat10b}). The integral with respect to $p$, which is understood in the principal value sense, is the filtering step (the Hilbert transform). The exterior sum is a quadrature rule corresponding to the backprojection integral.

%Let $\R^*$ be a suitably truncated adjoint transform. Due to the fact that $\R^*\R$ is a $\Psi$DO \bt{microlocally near the conormal bundle of $\sst\cap U$,} where $U$ is sufficiently small (this follows from assumptions~\ref{ass:Phi}, see e.g. \cite{qu-80, Beylkin1984}), we can restrict $\al$ to a small neighborhood $\Omega\ni\alst$ when applying $\R^*$. 

\bt{
To better understand \eqref{recon-orig}, we consider its continuous analogue. Suppose $w$ is the $\delta$-function. The continuous version of \eqref{recon-orig} reads 
\be\label{cont inv}
f^{\text{rec}}=\R^* (\CH\pa_p)\R f.
\ee
Here $\R^*$ is a weighted adjoint transform, and $\CH$ is the Hilbert transform acting with respect to $p$. By imposing additional restrictions on $\mps$, $\omega$, and $W$ we can ensure that $\R^* (\CH\pa_p)\R$ is a $\Psi$DO of order zero (see e.g. \cite{qu-80,Beylkin1984}) with some other desired properties (e.g., elliptic, principal symbol equal 1). 
%Assumptions~\ref{ass:Phi} imply that 
% . Moreover, if $W(y_0;\al’,p’)\not=0$, where $(\al’,p’)$ is such that $\s_{\al’,p’}$ is tangent to $\s$ at $y_0$, a judicious choice of $\omega$ ensures that the principal symbol of this $\Psi$DO is 1 microlocally near $(y_0,\vec\Theta_0)$. 
We do not do this, since our focus here is only the reconstruction of rapidly oscillating artifacts in $\frec$ away from $\s$. In particular, no attempt is made to achieve exact reconstruction. In view of this we impose only a minimal set of conditions that guarantee that Theorem~\ref{main res} holds. These conditions do not guarantee that $\R^* (\CH\pa_p)\R$ is a $\Psi$DO.}

Introduce the following functions:
%\be\label{two derivs}
%\dd_x\mps(\alst ,x_0)=u_0,\ \dd_\al(\mps(\alst ,x_0)-\mP(\alst ))=\mu_0,
%\ee
\be\label{psi simp}\begin{split}
\psi(\hat q):=&(1/2)\int_0^\infty w(\hat q+\hat p)\hat p^{-1/2}\dd \hat p,\ \bt{\hat q\in\br,}\\
\Psi(h;a,r):=&\sum_{\bt{k\in\mathbb Z}} \left[\psi\left(a(k-r)+h\right)-\psi\left(a(k-r)\right)\right],\ \color{black}h,a,r\in\br,a\not=0,\\ 
\color{black}\Psi(h;0,r):=&\bt{0,\ h,r\in\br},
\end{split}
\ee
and
\be\label{f jump}\begin{split}
\Delta f(y_0)=&\lim_{\e\to0^+}(f(y_0+\e\vec\Theta_0)-f(y_0-\e\vec\Theta_0)).
\end{split}
\ee
\bt{Various properties of $\psi$ and $\Psi$ (e.g., that $\psi$ is continuous and decays sufficiently fast, so that the series in the definition of $\Psi$ is absolutely convergent) are established in Sections~\ref{sec:lead term} and \ref{sec:PSI look}.}
  Our main result is as follows.

\begin{theorem}\label{main res}
Suppose $W\in C^\infty(\Omega\times\br\times U)$, and $\omega\in C^\infty(\Omega\times U’)$ for some small open sets $\Omega\ni\alst$, $U\ni y_0$, and $U’\ni x_0$. Under the assumptions~\ref{ass:Phi}, \ref{ass:f}, and \ref{ass:w}, one has 
\be\label{recon-2}\begin{split}
\e^{-1/2}&(\frec(x_0+\e\check x)-\frec(x_0))
=c\Psi(u_0\cdot\check x;\kappa\mu_0,\kst )+O(\e^{1/2}),\ \e\to0,\\
c:=&-\frac{\kappa\omega(\alst,x_0)W(\alst,\pst;y_0)}{\pi}(2/M)^{1/2}\, \Delta f(y_0),\\ 
\color{black}\kst:=&\color{black}\alst /\Delta\al,\ \kappa:=\Delta\al/\e, 
\color{black}
\end{split}
\ee
where $M$ is defined in \eqref{p der cond}, \bt{$u_0$ and $\mu_0$ are defined \eqref{two derivs 1},} and the $O(\e^{1/2})$ term is uniform with respect to $\check x$ confined to any bounded set.
\end{theorem}

\color{black}
To help the reader, we discuss various quantities occurring in \eqref{recon-2}.  
%The meaning of the quantities in \eqref{recon-2} is as follows.
\begin{enumerate}
\item $\check x$ is a rescaled displacement from a fixed point $x_0$ to a nearby reconstruction point $x$: $\check x=(x-x_0)/\e$;
\item For the classical Radon transform (CRT), $\mps(\al,x)=\vec\al\cdot x$, where $\al$ and $\vec\al$ are related by $\vec\al=(\cos\al,\sin\al)$;
\item $\alst,\pst$ are the values such that the integration curve $\sst=\s_{\alst,\pst}$ contains $x_0$ is tangent to $\s$ at some point, denoted $y_0$ (see Figure~\ref{fig:assns});
\item $W(\al,p;y)$ and $\omega(\al,x)$ are integration weights in $\R$ and its adjoint (see \eqref{grt def}, \eqref{recon-orig}, \eqref{cont inv}, and the discussion around the latter equation). For the CRT, $W(\al,p;y)\equiv 1$ and $\omega(\al,x)\equiv1$;
\item $\kappa=\Delta\al/\e$, where $\Delta\al$ is the step-size along $\al$;
\item Up to a nonzero factor, $M$ is the difference of curvatures of $\s$ and $\sst$ at $y_0$;
\item $\Delta f(y_0)$ is the value of the jump of $f$ across $\s$ at $y_0$;
\item $\kst=\alst/\Delta\al$ is the “index” value corresponding to the angle $\alst$. We put the word index in quotes, because $\kst$ is not necessarily an integer. As is easily seen from \eqref{psi simp} and \eqref{recon-2}, only the fractional part of $\kst$ is important;
\item The quantities $u_0:=\dd_x\mps(\alst ,x_0)$ and $\mu_0:=\pa_\al(\mps(\alst ,x_0)-\mP(\alst ))$ depend on the properties of the Radon transform (via the function $\mps$) and the curve $\s$. For the CRT, $u_0=\vec\al_\star$ and $\mu_0=\vec\al_\star^\perp\cdot(x_0-y_0)$, so $|\mu_0|=|x_0-y_0|$.
\end{enumerate}

The following corollary, which follows immediately from Theorem~\ref{main res}, illustrates what eq. \eqref{recon-2} looks like in the case of the classical Radon transform.

\begin{corollary}\label{cor:CRT}
Let $\R$ be the classical Radon transform. Under the assumptions of Theorem~\ref{main res} one has 
\be\label{recon-2 CRT}\begin{split}
\e^{-1/2}&(\frec(x_0+\e\check x)-\frec(x_0))
=c\Psi(\alst\cdot\check x;\kappa\vec\al_\star^\perp\cdot(x_0-y_0),\kst )+O(\e^{1/2}),\ \e\to0,\\
c:=&-(\kappa/\pi)(2r)^{1/2}\, \Delta f(y_0),\
\kst:=\alst /\Delta\al,\ \kappa:=\Delta\al/\e, 
\end{split}
\ee
where $r$ is the radius of curvature of $\s$ at $y_0$, and the $O(\e^{1/2})$ term is uniform with respect to $\check x$ confined to any bounded set.
\end{corollary}

See Section~\ref{sec:numerics} for more details on how to apply \eqref{recon-2} for the classical and circular Radon transforms. 

\color{black}

\section{Proof of Theorem~\ref{main res}}\label{sec:proof}

By \eqref{two derivs 1}, $u_0\not=0$, $\mu_0\not=0$. 
By linearity of the Radon transform, we can assume that the support of $f$ is contained in a small neighborhood of $y_0$ (i.e., by shrinking $U$ as much as necessary). By assumption~\ref{ass:Phi}(5), shrinking $U$ and $\Omega$ even more,  we can assume that there exists $c>0$ such that  
\be\label{p-der bnd}\begin{split}
&\s_{\al,p}\cap U=\varnothing\text{ for any } \al\in\Omega, |p-\pst|\ge c.
\end{split}
\ee
Then
\be\label{sing descr}
\hat f(\al,p)=\varphi_1(\al)(p-\mP(\al))_+^{1/2}+\varphi_2(\al,p)(p-\mP(\al))_+^{3/2}+\varphi_3(\al,p),\
\al\in\Omega,p\in\br,
\ee
where $\varphi_1\in \coi(\Omega)$, $\varphi_{2,3}\in C^\infty(\Omega\times\br)$, and
\be\label{phi1}
\varphi_1(\alst )=W(\alst,\pst;y_0)\Delta f(y_0)2(2/M)^{1/2}.
\ee
For the classical Radon transform this result is established in \cite{rz1, rz2}. For the generalized Radon transform it easily follows from $\dd_x\mps(\alst,y_0)\not=0$ and $M\not=0$ (see assumptions~\ref{ass:Phi}(1, 4)) by applying the method of proof of Lemma~3.5 in \cite{Katsevich2020a}. 
%\bt{The assumption $M\not=0$ is used in the derivation of \eqref{phi1} in an essential way.}

Since $f(x)$ is compactly supported, $\hat f(\al, p)$ is compactly supported in $p$ by \eqref{p-der bnd}. Hence we can assume that $\varphi_2(\al,p)$ is compactly supported as well, and 
\be\label{f3-prop}
\varphi_3(\al,p)\equiv -\varphi_1(\al)(p-\mP(\al))_+^{1/2},\ \al\in\Omega,|p|\ge c,
\ee
for some $c$.

The idea of the proof is to split $\hat f$ into three terms using \eqref{sing descr}, substitute each of them one by one into \eqref{data_eps}, \eqref{recon-orig}, and investigate the resulting expressions separately.

\subsection{Beginning of proof. Estimate of the leading term.}\label{sec:lead term}

Replace $\hat f(\al,s)$ with $\varphi_1(\al)(s-\mP(\al))_+^{1/2}$ in \eqref{data_eps} and substitute into \eqref{recon-orig}. After simple transformations we get
\be\label{recon-1}\begin{split}
\freco(x):=&-\frac{\Delta\al}{2\pi \e^{1/2}}\sum_{\al_k\in\Omega} \omega(\al_k,x)\varphi_1(\al_k) \psi((\mps(\al_k,x)-\mP(\al_k))/\e),\\ 
\psi(\hat q):=&(2\pi)^{-1}\int (\hat p-\hat q)^{-1}\int w(\hat p-\hat s)\hat s_+^{-1/2}\dd\hat s\dd \hat p.
\end{split}
\ee
After additional transformations with the help of the integral \eqref{key int}, $\psi$ simplifies to the expression in \eqref{psi simp}. \bt{These transformations are justified by applying $\psi$ in \eqref{recon-1} to a test function and changing the order of integration using the result in \cite[Section III.28.4]{Muskh1958}.} In turn, \eqref{psi simp} gives 
\be\label{psi props}\begin{split}
&\psi(\hat q)=0,\ \hat q>c;\quad
\bt{\psi(\hat q)\ \text{ is uniformly continuous on $\br$};}\\ 
&\psi^{(n)}(\hat q)=c_n(-\hat q)^{-(1/2)-n}+O(|\hat q|^{-(3/2)-n}),\ \hat q<-c,\hat q\to-\infty,n=0,1,2,\dots,\\
\end{split}
\ee
for some $c>0$ and $c_n$. \bt{Since $\hat p_+^{-1/2}\in L_{loc}^{q’}(\br)$ for any $q’<2$, Assumption~\ref{ass:w}(1) and \cite[Exercise 11, p. 196]{Rudin_real} imply that $\psi$ is uniformly continuous on $\br$. Note that $\psi(\hat q)$ is of limited smoothness on a compact set, outside of which $\psi$ is $C^\infty$.}

%From \eqref{recon-1},
%\be\label{recon-2 aux}\begin{split}
%\freco(x)-&\freco(x_0)=O(\e)-\frac{\Delta\al}{2\pi\e^{1/2}}\sum_{\al_k\in\Omega} \omega(x_0,\al)\varphi_1(\al)\\
%&\times 
%\left[\psi\left(\frac{\mps(\al,x)-\mP(\al)}\e\right)-\psi\left(\frac{\mps(\al,x_0)-\mP(\al)}\e\right)\right],\ \al=\al_k.
%\end{split}
%\ee
%The first $O(\e)$ term on the right in \eqref{recon-2} denotes the contribution, which arises due to the $x$-dependence of $\om$ in \eqref{recon-1}. Here we use that $|\om(x_0+\e\check x,\al)-\om(x_0,\al)|\le c\e$ and $|\varphi_1(\al)|,|\psi(\hat q)|\le c$ for some $c$ and all $\al\in\Omega$, $\hat q\in\br$.

Using the notation in \eqref{two derivs 1} and \eqref{psi simp} we formulate the following result
\begin{lemma}\label{lem:first res} Under the assumptions of Theorem~\ref{main res} one has
\be\label{recon-2 lem}\begin{split}
\e^{-1/2}(\freco(x_0+\e\check x)-&\freco(x_0))\\
=&-\frac{\kappa\omega(\alst,x_0)\varphi_1(\alst )}{2\pi}\Psi(u_0\cdot\check x;\mu_0\kappa,\kst )+O(\e^{1/2}),
\end{split}
\ee
where the $O(\e^{1/2})$ term is uniform with respect to $\check x$ confined to any bounded set.
\end{lemma}

\bt{The proof of the lemma is in subsection~\ref{ssec:main term}.}

\subsection{The second term}
Similarly, replace $\hat f(\al,s)$ with $\varphi_2(\al,s)(s-\mP(\al))_+^{3/2}$ in \eqref{data_eps} and substitute into \eqref{recon-orig}. After simple transformations we get with some $c$
\be\label{recon-t2}\begin{split}
\frect(x):=&c\e\sum_{\al_k\in\Omega} \omega(\al_k,x)g_2(\mps(\al_k,x),\mP(\al_k),\al_k),\ x=x_0+\e\check x,\\ 
g_2(p,q,\al):=&\int (t-p)^{-1}\pa_t\int w_\e(t-s) \varphi_2(\al,s)(s-q)_+^{3/2}\dd s \dd t,\\
&\bt{p=\mps(\al_k,x),\ q=\mP(\al_k).}
\end{split}
\ee
\bt{Therefore, in \eqref{recon-t2} $p,q$ satisfy
\be\label{pq bnd}\begin{split}
|p|\le\sup_{\al\in\Omega}|\mps(\al,x_0)|+O(\e),\ |q|\le c,
\end{split}
\ee
where $c$ is the same as in \eqref{p-der bnd}. Reducing, if necessary, $\Omega$ further, we can assume that the supremum in \eqref{pq bnd} is bounded. Thus, $|p|,|q|\le P$ for some $P<\infty$.} For simplicity, the dependence of $g_2$, $\varphi_2$, and related functions on $\al$ will be omitted from notation. Rewrite $g_2$ as follows:
\be\label{g alt}
g_2(p,q)=\int w_\e(p-t) \int  (s-t)^{-1}\pa_s \left(\varphi_2(s)(s-q)_+^{3/2}\right)\dd s \dd t.
\ee
Using the results in \cite[\S 8.3]{gakhov}, we find
\be\label{g alt v2}\begin{split}
g_2(p,q)&=\int w_\e(p-t) \left(\varphi_{2,1}(t,q)(t-q)_-^{1/2}+\varphi_{2,2}(t,q)\right) \dd t
\end{split}
\ee
for some smooth and bounded $\varphi_{2,1}$ and $\varphi_{2,2}$. The same result can be obtained by elementary means by writing 
\be\label{key int cor}
\int_0^\infty \frac{\varphi(s)}{s^{1/2}(s-\rho)}\dd s=
\varphi(\rho)\int_0^\infty \frac{\dd s}{s^{1/2}(s-\rho)}+\int_0^\infty \frac{\varphi(s)-\varphi(\rho)}{s-\rho}s^{-1/2}\dd s,
\ee
using the integral (see \cite[Equations 2.2.4.25 and 2.2.4.26]{pbm1})
\be\label{key int}
\int (s-\rho)^{-1}s_+^{-1/2}\dd s=\pi \rho_-^{-1/2},
\ee
and substituting $\rho=t-q$, $\varphi(s)=s\bigl[(3/2)\varphi_2(s+q)+s\varphi_2’(s+q)\bigr]$.

From \eqref{g alt v2} it follows that
\be\label{g2 diff}\begin{split}
&|g_2(p+\Delta p,q)-g_2(p,q)|\\
&\le O(|\Delta p|)+c\max_{|\tau|\le c\e}\left|(p-q+\Delta p+\tau)_-^{1/2}-(p-q+\tau)_-^{1/2}\right|
\end{split}
\ee
for some $c$. Recall that in \eqref{g2 diff} 
\be\label{aux vars}
p-q=\mps(\al_k,x_0)-\mP(\al_k),\ \Delta p=\mps(\al_k,x)-\mps(\al_k,x_0)=O(\e),\ \tau=O(\e), 
\ee
where $x=x_0+\e\check x$. Since $\mu_0\not=0$ (cf. \eqref{two derivs 1}), we have $|\mps(\al,x_0)-\mP(\al)|\ge c|\al-\alst |$ for any $\al\in\Omega$ and some $c>0$. Therefore, there exists $c_1>0$ such that whenever $|\al-\alst |\ge c_1\e$
%\be\label{two ineqs}
%|p_*-q+\Delta p_*+p|\ge c\e\text{ or } |p_*-q+p|\ge c\e,
%\ee
and $\e>0$ is sufficiently small, the expressions $(p-q+\Delta p+\tau)_-^{1/2}$ and $(p-q+\tau)_-^{1/2}$ are either both zero or both nonzero. When they are both nonzero, the magnitude of their difference equals
\be\label{brack diff}
\frac{|\Delta p|}{|p-q+\Delta p+\tau|^{1/2}+|p-q+\tau|^{1/2}}
\le \frac{c\e}{|\al-\alst|^{1/2}},\quad |\al-\alst |\ge c_1\e,
\ee
for some $c$. Also, there are finitely many $k$ (close to $\kst $) such that $|\al_k-\alst |< c_1\e$. For those $k$, the same difference is $O(\e^{1/2})$. 

Using \eqref{g2 diff} and \eqref{brack diff} in \eqref{recon-t2}, we find similarly to \eqref{recon-2}:
\be\label{recon-t2-2}\begin{split}
&\e^{-1/2}(\frect(x_0+\e\check x)-\frect(x_0))\\
&=O(\e^{1/2})+O(\e^{1/2})\biggl[\e^{1/2}+\sum_{1\le k\le O(1/\e)} \frac{\e}{(k\e)^{1/2}}\biggr]
=O(\e^{1/2}).
\end{split}
\ee
The first $O(\e^{1/2})$ term on the right in \eqref{recon-t2-2} absorbs the contributions, which arise due to the $x$-dependence of $\om$ in \eqref{recon-t2} and due to the $O(|\Delta p|)=O(\e)$ term in \eqref{g2 diff}. Here we use that $|\om(x_0+\e\check x,\al)-\om(x_0,\al)|\le c\e$ and $|g_2(p,q,\al)|\le c$ for some $c$ and all $\al\in\Omega$, $|p|,|q|\le P$. 
%Recall that $\varphi_2$ is compactly supported.

\subsection{The third term}
Finally, replace $\hat f(\al,s)$ with $\varphi_3(\al,s)$ in \eqref{data_eps} and substitute into \eqref{recon-orig}. Recall that $\varphi_3$ is not necessarily compactly supported in $s$ (cf. \eqref{f3-prop}), and 
\be\label{f3 props}
\pa_s^l \varphi_3(\al,s)=O(|s|^{(1/2)-l}),\ s\to\infty,\ \al\in\Omega,\ l=0,1,2,
\ee
where the big-$O$ term is uniform in $\al$. Similarly to \eqref{recon-t2} and \eqref{g alt}, we find
\be\label{recon-t3}\begin{split}
\frecth(x):=&c\e\sum_{\al_k\in\Omega} \omega(\al_k,x)g_3(\mps(\al_k,x),\al_k),\\ 
g_3(p,\al):=&\int (t-p)^{-1}\pa_t\int w_\e(t-s) \varphi_3(\al,s)\dd s \dd t\\
=&\int w_\e(-\tau) \int  s^{-1}\pa_s \varphi_3(\al,s+\tau+p) \dd s \dd \tau,\ |p|\le P.
\end{split}
\ee

The following lemma is proven in appendix~\ref{sec:ramp phi}.
\begin{lemma}\label{lem:diff phi3} One has
\be\label{diff est 3}
\int  s^{-1}\pa_s \left[\varphi_3(\al,s+q+\Delta q)-\varphi_3(\al,s+q)\right] \dd s
=O(|\Delta q|),\ \Delta q\to0,
\ee
uniformly in $\al\in\Omega$, $|q|\le c$, for any $c$.
\end{lemma}

Using Lemma~\ref{lem:diff phi3}, the analogue of \eqref{g2 diff} becomes (with $\Delta p$ the same as in \eqref{aux vars})
%(omitting the dependence on $\al$ for convenience)
\be\label{g3 diff}\begin{split}
&|g_3(p+\Delta p,\al)-g_3(p,\al)|\\
&\le c\max_{|\tau|\le c\e}\biggl|\int s^{-1}\pa_s  \left[\varphi_3(\al,s+\tau+p+\Delta p)-\varphi_3(\al,s+\tau+p)\right]\dd s\biggr|\\
&=O(|\Delta p|)=O(\e),\ \al\in\Omega,
\end{split}
\ee
for some $c$. Hence, we obtain similarly to \eqref{recon-t2-2}:
\be\label{recon-t3-diff}
\e^{-1/2}(\frecth(x_0+\e\check x)-\frecth(x_0))=O(\e^{1/2}).
\ee
Combining \eqref{recon-2}, \eqref{phi1}, \eqref{recon-t2-2}, \eqref{recon-t3-diff}, and using that $\frec=\freco+\frect+\frecth$, we finish the proof of Theorem~\ref{main res}.

%\newpage
\section{A more detailed look at the function $\Psi$}\label{sec:PSI look}

\subsection{Properties of the function $\Psi$}\label{sec:PSI props}
Theorem~\ref{main res} shows that the function $\Psi$ defined in \eqref{psi simp} plays a key role in the description of the aliasing artifact. \bt{By \eqref{psi props}, the series that defines $\Psi$ converges absolutely at every point.} Here we prove some of the properties of $\Psi$. 
%In this section $\kst$ denotes any real number rather than the specific value defined in \eqref{psi simp}.

%We begin by stating some useful properties of $\Psi$:
\begin{lemma}\label{lem:PSI props}
Under the assumptions~\ref{ass:w} one has 
\begin{enumerate}
\item \bt{$\Psi$ is continuous on $\br\times(\br\setminus 0)\times\br$;}
%on $\br\times(\br\setminus\{0\})\times\br$.}
\item $\Psi(h;a,r+1)=\Psi(h;a,r)$ and $\bt{\Psi(h;-a,-r)=\Psi(h;a,r)}$ for all $h,a,r \in\br$;
\item $\Psi(h+a;a,r)=\Psi(h;a,r)$ for all $h,a,r \in\br$;
\end{enumerate}
%\be\label{PSI props}\begin{split}
%&\Psi(h;a,r+1)=\Psi(h;a,r),\ \bt{\Psi(h;-a,-r)=\Psi(h;a,r)},\\
%&\Psi(h+a;a,r)=\Psi(h;a,r)\ \forall h,a,r \in\br.
%\end{split}
%\ee
\end{lemma}
\begin{proof}
\bt{
%Here we consider the case when $a\not=0$. The continuity at $a=0$ follows from Lemma~\ref{lem:PSI decay} with $N=1$. 
When $a$ is bounded away from zero, the number of terms with limited smoothness in the sum in \eqref{psi simp} is uniformly bounded when $h$ and $r$ are confined to a bounded set. Hence we can represent $\Psi$ as a sum of finitely many continuous terms and an absolutely convergent series, whose terms are smooth functions. This proves statement (1).}

The first half of statement (2) is obvious. \bt{The second half of statement (2) follows immediately by replacing $a\to-a$, $r\to-r$ in \eqref{psi simp}, and changing the index of summation $k\to-k$.} 

To prove statement (3), fix some $c\gg 1$ and shift the index of summation $k’=k+1$ in \eqref{psi simp}:
\be\label{PSI per st1}
\Psi(h+a;a,r)=\sum_{k’\le c} \left[\psi\left(a(k’-r)+h\right)-\psi\left(a(k’-1-r)\right)\right].
\ee
At first glance, to finish the proof we can just change back $k=k’-1$ in the second $\psi$. This does not work, since each of the sums taken separately is divergent (cf. \eqref{psi props}). Hence we argue differently. We have for any $K\gg1$:
\be\label{PSI per st2}\begin{split}
\Psi(h+a;a,r)=&\sum_{k’=-K}^c \left[\psi(a(k’-r)+h)-\psi(a(k’-1-r))\right]+O(K^{-1/2})\\
=&\sum_{k’=-K}^c \left[\psi(a(k’-r)+h)-\psi(a(k’-r))\right]-\psi(a(K-1-r))\\
&+O(K^{-1/2})=\Psi(h;a,r)+O(K^{-1/2}),\ K\to\infty.
\end{split}
\ee
The desired assertion now follows.
\end{proof}

\color{black}
\begin{lemma}\label{lem:PSI decay}
\bt{Suppose $w$ is compactly supported and $w^{(N)}\in L^q(\br)$ for some $N\ge 1$ and $q>2$. One has:}
\be\label{PSI asymp}
\pa_h^{n_1}\pa_r^{n_2}\Psi(h;a,r)=O(|a|^{{N-(n_1+n_2)}}),\ a\to0,\ n_1,n_2\ge 0,n_1+n_2\le N-1,
\ee
uniformly in $h,r\in\br$. 
\end{lemma}
\begin{proof} We need the following simple lemma, which follows immediately from the Euler-MacLaurin summation formula \cite[eq. (25.7)]{Kac2002}. 
%applied to the entire line \cite[Theorem 5.5]{Atkinson1989}. 
For convenience of the reader, the lemma is proven in appendix~\ref{sec:aux conv}. 

%\newpage
%\be\label{EML}\begin{split}
%\sum_{k=a}^{b-1} f(k)=&\int_a^b f(t)\dd t+\sum_{n=1}^N\frac{b_n}{n!}(f^{(n-1)}(b)-f^{(n-1)}(a))\\
%&-\int_a^b \frac{B_N(\{1-t\})}{N!} f^{(N)}(t)\dd t.
%\end{split}
%\ee

\begin{lemma}\label{lem:aux conv}
\bt{Pick some $N’\ge 1$.} Suppose $g,g^{(N’)}\in L^1(\br)$, $g^{(n)}(t)\to0$ as $t\to\infty$ for any $n=0,1,2,\dots,N’-1$, and $\int_{\br} g(x)dx=0$. Then, 
\be\label{riem sum lim}
\bigl|\e\sum_{k\in\mathbb Z} g(\e k)\bigr|\le c\,\e^{N’}\Vert g^{(N’)}\Vert_{L^1(\br)}
\ee
for some $c$ independent of $g$ and $\e$.
\end{lemma}

%\begin{lemma}\label{lem:aux conv}
%\bt{Pick some $N\ge 1$.} Suppose $g^{(n)}\in L^1(\br)$ and $g^{(n)}(t)\to0$ as $t\to\infty$ for any $0\le n\le N$, and $\int_{\br} g(x)dx=0$. Then, 
%\be\label{riem sum lim}
%\bigl|\e\sum_{k\in\mathbb Z} g^{(n)}(\e (k-r))\bigr|\le c_{n,N}\e^{N-n}\Vert g^{(N-n)}\Vert_{L^1(\br)},\
%0\le n\le N-1,
%\ee
%for some $c_{n,N}$ independent of $g$ and $r\in\br$.
%\end{lemma}

Set 
\be\label{del psi}
g(t):=\pa_h^{n_1}\pa_r^{n_2}(\psi(t-ar+h)-\psi(t-ar)). 
\ee
The dependence of $g$ on $h$ and $r$ is omitted for simplicity. As is easily seen, $g$ satisfies the assumptions of Lemma~\ref{lem:aux conv}. Indeed, due to Lemma~\ref{lem:PSI props}(2,3), we can assume $h\in[0,a)$, $r\in [0,1)$. 
The assumption $w^{(N)}\in L^q(\br)$, $q>2$, and \eqref{psi simp} imply that all the derivatives of $\psi$ up to the order $N$ are continuous on $\br$. 

From \eqref{psi props}, $|g^{(m)}(t)|\le c_m(1+|t|)^{-3/2}$, $0\le m\le N-(n_1+n_2)$, for some $c_m$ independent of $h$ and $r$. Hence $g$ decays sufficiently fast at infinity. 

It remains to check that $g$ integrates to zero. If $n_1>0$ or $n_2>0$, this is obvious. Suppose $n_1=n_2=0$. For some $c>0$,
\be\begin{split}
\int_{\br} g(t)\dd t
&=\int_{-A}^c g(t)\dd t+O(A^{-1/2})\\
&=-\int_{-A}^{h-A} \psi(t)\dd t+O(A^{-1/2})=O(A^{-1/2}),\ A\to\infty.
\end{split}
\ee
Application of Lemma~\ref{lem:aux conv} to $g$ in \eqref{del psi} with $\e=a$ and $N_1=N-(n_1+n_2)$ proves the desired assertion. The uniformity with respect to $h$ and $r$ is obvious. 
%from the proof of Lemma~\ref{lem:aux conv}.
% (by noticing that the estimates for $A_j(\e)$ are uniform with respect to $h$).
\end{proof}

%If moother than Theorem~\ref{main res} requires, then $\Psi$ has additional properties. 
%
%$\pa_h^N\Psi(h;a,r)$ is continuous on $\br\times(\br\setminus\{0\})\times\br$. \red{Maybe even at $a=0$?}The proof of the first assertion is similar to that of the first assertion of Lemma~\ref{lem:PSI props}.To prove the second assertion, 

\begin{corollary} Suppose $w$ is compactly supported, and $w^{(N)}\in L^q(\br)$ for some $N\ge 1$ and $q>2$. Then the derivatives $\pa_h^{n_1}\pa_r^{n_2}\Psi(h;a,r)$, $n_1,n_2\ge 0$, $n_1+n_2\le N-1$, are continuous for all values of their arguments.
\end{corollary}
\begin{proof} The continuity away from $a=0$ is proven the same way as assertion (1) of Lemma~\ref{lem:PSI props}. The continuity at $a=0$ follows from Lemma~\ref{lem:PSI decay}. 
\end{proof}

\color{black}
\subsection{Computing $\Psi$ numerically}\label{sec:PSI numer}
Numerically, we compute $\Psi$ using the following approach. \bt{Due to Lemma~\ref{lem:PSI props}(2,3), we assume $h\in[0,a)$, $r\in [0,1)$.} The mollifier in our experiments is given by
\be\label{apert fn}
w(t)=(15/16)(1-t^2)_+^2.
\ee
First, $\psi(t)$ is computed by analytically evaluating the integral in \eqref{psi simp}. Then we compute $\Delta\psi(t,h):=\psi(t+h)-\psi(t)$. For moderate values of $t$ we compute $\Delta\psi$ directly from the definition. For $t\ll -1$ we use 
\be\label{del psi asymp}
\Delta\psi(t,h)\approx h / (4|t|^{3/2}). 
\ee
Finally, we write
\be\label{PSI numer}\begin{split}
\Psi(h;a,r)\approx &\sum_{k=-K+1}^c \Delta\psi(a(k-r),h)+\frac{h}{4|a|^{3/2}}\sum_{k=K}^{\infty}k^{-3/2},
\end{split}
\ee
where $c>0$ is selected so that $\Delta\psi(a(k-r),h)=0$ for all $k>c$ and $h\in[0,a)$, and $K\gg 1$. The last sum is estimated using the asymptotic formula for the Hurwitz Zeta Function \cite[Equation (1.1)]{Nemes2017}
\be\label{Hurw}
\zeta(s,t):=\sum_{k=0}^\infty (k+t)^{-s}=\frac{t^{1-s}}{s-1}+\frac{t^{-s}}2+O(t^{-(s+1)}),\ t\to+\infty,
\ee
where $s=3/2$ and $t=K$. The plots of $\Psi(ah’;a,r)$, $0\le h’\le1$, for the values $a=1,2,4$ and $r=1/3$ are shown in Figure~\ref{fig:PSI}. 

\begin{figure}[h]
{\centerline{
{\epsfig{file={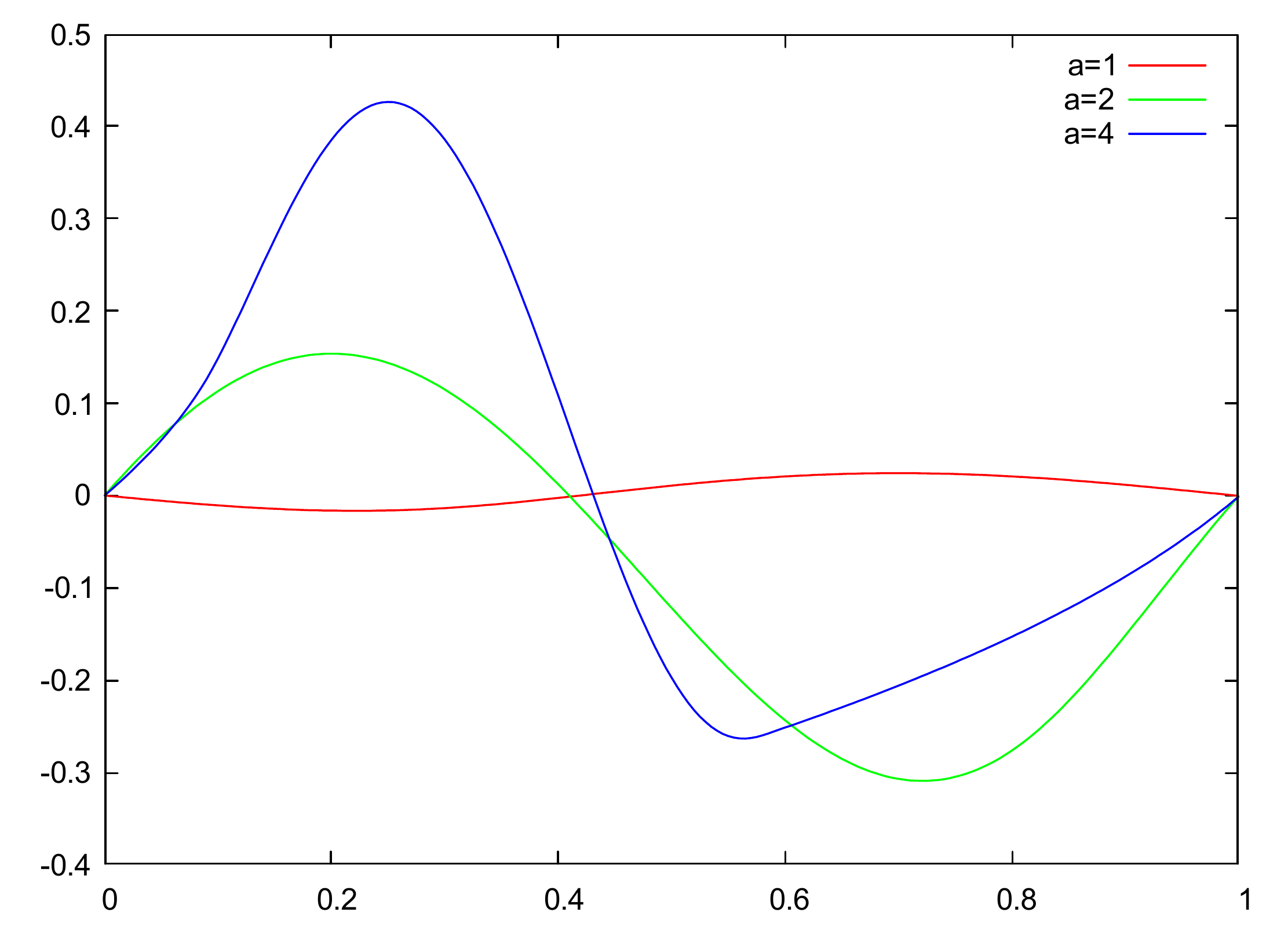}, width=7cm}}
}}
\caption{Plots of $\Psi(ah’;a,r)$ for three values of $a$. The variable $h’$ is on the horizontal axis.}
\label{fig:PSI}
\end{figure}

In agreement with Lemma~\ref{lem:PSI decay}, we see that $\Psi(ah’;a,r)$ decays rapidly as $a\to0$.

\section{Numerical experiments}\label{sec:numerics}

\subsection{Classical Radon transform}

In this subsection we experiment with the classical Radon transform (CRT), which integrates over lines:
\be\label{CRT example}
\hat f(\al,p)=\int_{S_{\al,p}}f(x)\dd x,\ \vec\al=(\cos\al,\sin\al),\ S_{\al,p}:=\{x\in\br^2:\al\cdot x=p\}.
\ee
Reconstruction uses \eqref{recon-orig}:
\be\label{CRT data}\begin{split}
\frec(x)=&-\frac{\Delta\al}{2\pi}\sum_{|\al_k|\le \pi/2} \frac{1}\pi \int \frac{\pa_p\hat f_\e(\al_k,p)}{p-\mps(x,\al_k)}\dd p,\ \mps(x,\al)\equiv \vec\al\cdot x,\\
\hat f_\e(\al_k,p)=&\int w_\e(p-\rho)\hat f(\al_k,\rho)\dd\rho,\ \al_k=-(\pi/2)+(\pi/N_{\al})(k+\de),
\end{split}
\ee
and $w$ is the same as in \eqref{apert fn}. The weights in both the Radon transform and the inversion formula are set to 1: $W(\al,p;x)\equiv1$, $\omega(\al,x)\equiv1$.

The function $f$ is the characteristic function of the disk centered at the origin with radius $r$. Thus, $\s=\{x\in\br^2:|x|=r\}$. \bt{By \eqref{Phi eqs}, 
\be
|\dd H(y_0)|=|\dd_x\mps(\alst,y_0)|=|\alst|=1.
\ee
Therefore,} by \eqref{p der cond}, 
\be\label{crt aux}
M=-(\vec\Theta_0^\perp\cdot\pa_y)^2 H(y)|_{y=y_0}=1/r>0
\ee
is the curvature of $\s$ at $y_0$. Also, $\vec\Theta_0=\dd H(y_0)$ points towards the center of curvature of $\s$ at $y_0$ \bt{(the center of the disk)}.

At a given $x\not\in\s$, aliasing arises due to the parts of $\s$ where the lines $\s_{\al,p}\ni x$ are tangent to $\s$. For $|x|>r$, two such lines exist. We pick $x_0=(r,b)$ and find two pairs $(\alst ,\pst)$ with the required properties. Clearly, one of the pairs is $(\alst =\pi,\pst=-r)$, and the other - $(\alst =2\tan^{-1}(b/r)-\pi,\pst=-r)$. \bt{This choice of values of $(\alst ,\pst)$ ensures that $\s_{\alst,\pst+\de}$, where $0<\de\ll1$, intersects $\s$ at two points (cf. the paragraph following \eqref{p der cond}). See Figure~\ref{fig:CRT_diagr}, where the first pair (with $\alst=\pi$) is shown in red, and the second - in blue. Contributions coming from a neighborhood of each point of tangency $y_0$ are computed by \eqref{recon-2} using the corresponding values of parameters (computed elsewhere in this subsection) and added.} For reconstructions we use $r=5$ and $x_0=(5,7)$. To better illustrate the aliasing artifact we also reconstruct a small region of interest (ROI), which is a square centered at $x_0$ with side length $40\e$.

\begin{figure}[h]
{\centerline{
{\epsfig{file=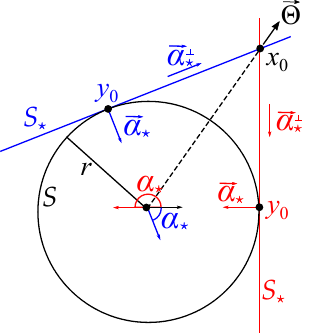}, height=4.5cm}}
}}
\caption{\bt{Illustration of various quantities used in the main formula \eqref{recon-2} to predict aliasing from a disk in the case of the classical Radon transform.}}
\label{fig:CRT_diagr}
\end{figure}

For computations we also need $u_0$ and $\mu_0$ (cf. \eqref{recon-2}). They follow easily from \eqref{two derivs 1}:
\be\label{u0 mu0}
u_0=\vec\al_\star,\ \mu_0=\vec\al_\star^\perp\cdot(x_0-y_0),
\ee
where $y_0$ is the point where $\sst$ is tangent to $\s$. \bt{As is seen from Figure~\ref{fig:CRT_diagr}, $\mu_0=-|x_0-y_0|$ for the first (red) pair $(\alst,\pst)$, and $\mu_0=|x_0-y_0|$ for the second (blue) pair.}

\begin{figure}[h]
{\centerline{
{\hbox{
{\epsfig{file={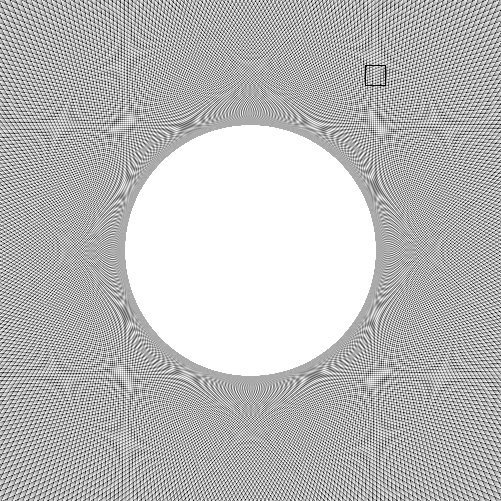}, height=4.5cm}}
{\epsfig{file={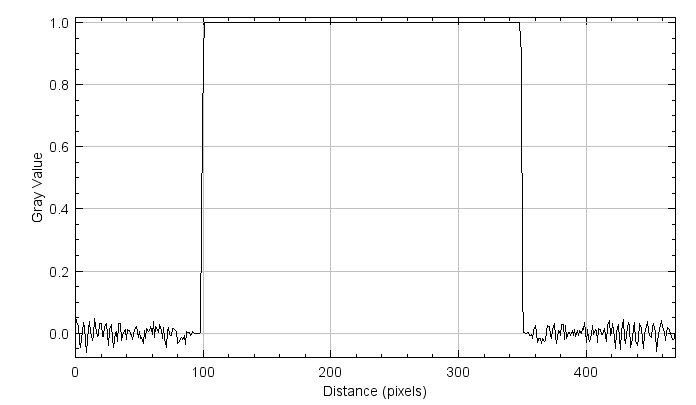}, height=4.5cm}}
}}}}
\caption{CRT reconstruction of the region $|x_1|,|x_2|\le 10$: $\e=0.02$, $N_{\al}=200$, $\de=0.03$. Left: global reconstruction, right: profile of the reconstruction through the center.}
\label{fig:global_CRT_N=200}
\end{figure}

\begin{figure}[h]
{\centerline{
{\hbox{
{\epsfig{file={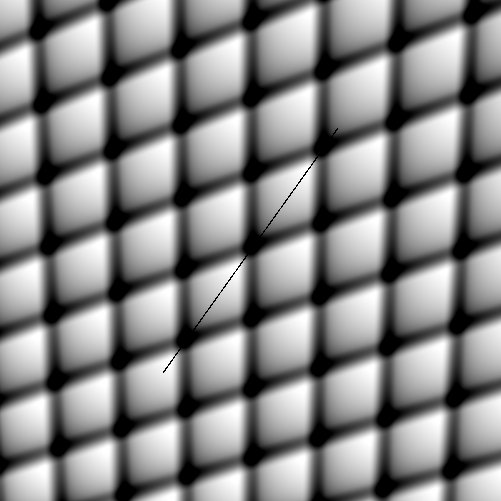}, height=4.5cm}}
{\epsfig{file={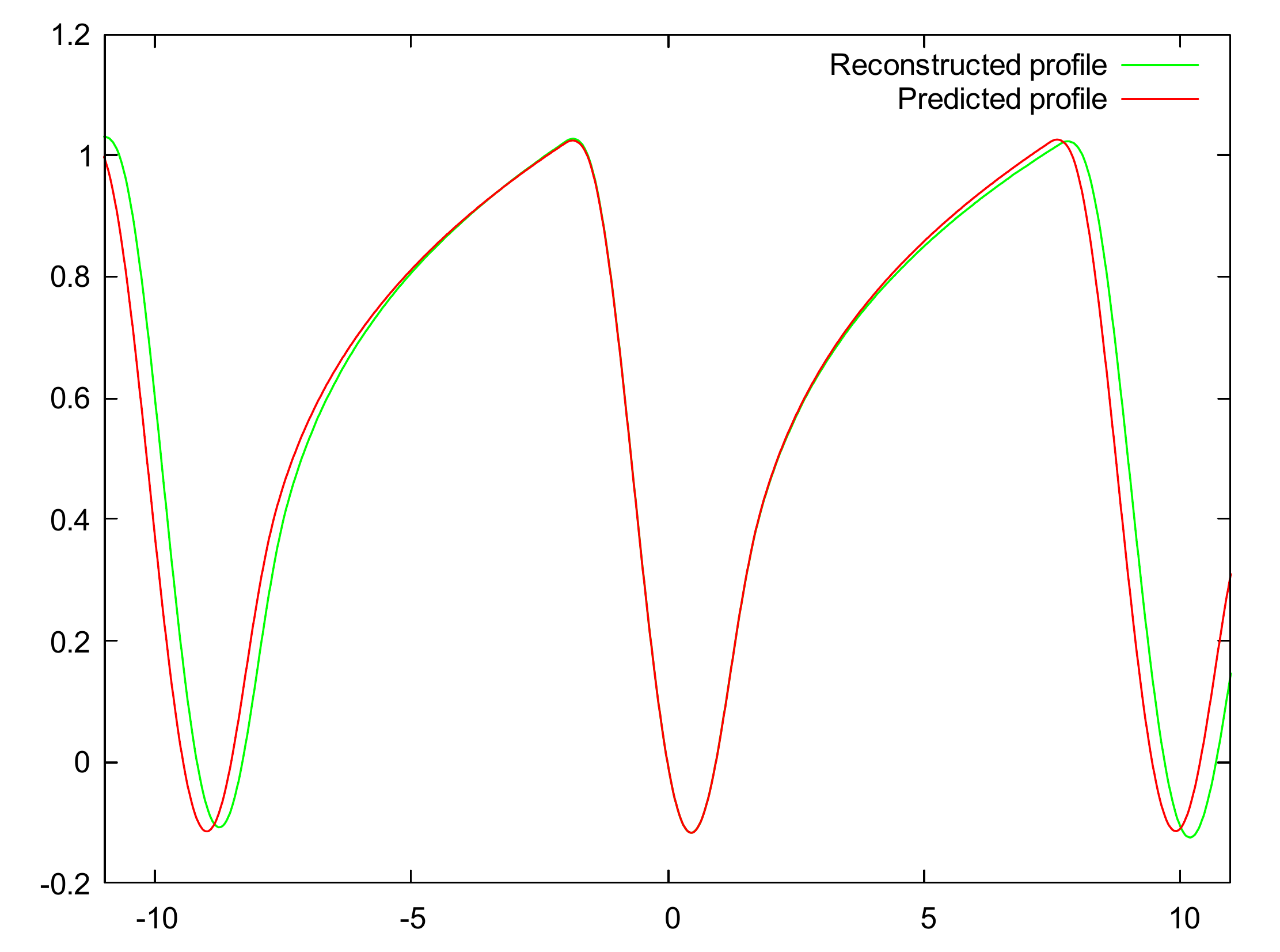}, height=4.5cm}}
}}}}
\caption{ROI CRT reconstruction: $\e=0.02$, $N_{\al}=200$, $\de=0.03$. The ROI is the square shown in Figure~\ref{fig:global_CRT_N=200}. Left: reconstructed ROI, right: reconstructed (green) and predicted (red) profiles along the line segment $x=x_0+\e h\vec\Theta$, $|h|\le 11$, shown on the left. The variable $h$ is on the horizontal axis.}
\label{fig:local_CRT_N=200_shift=0.03}
\end{figure}

\begin{figure}[h]
{\centerline{
{\hbox{
{\epsfig{file={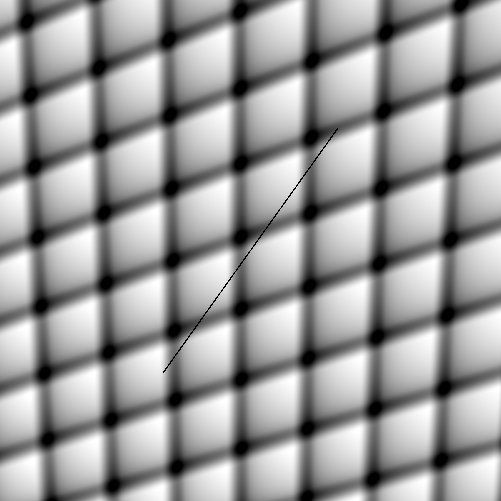}, height=4.5cm}}
{\epsfig{file={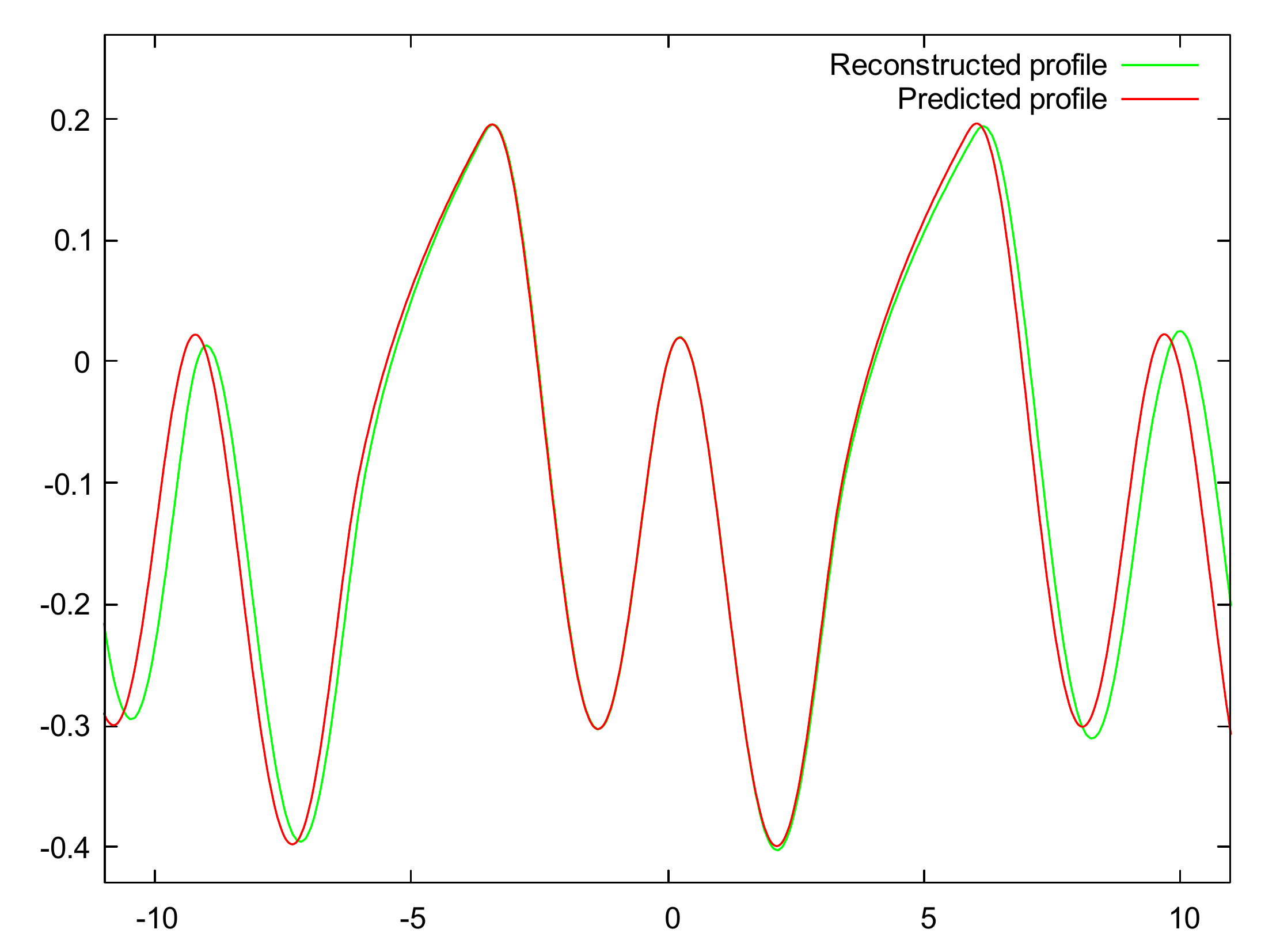}, height=4.5cm}}
}}}}
\caption{ROI CRT reconstruction: $\e=0.02$, $N_{\al}=200$, $\de=0.2$. The ROI is the square shown on the left in Figure~\ref{fig:global_CRT_N=200}. Left: reconstructed ROI, right: reconstructed (green) and predicted (red) profiles along the line segment $x=x_0+\e h\vec\Theta$, $|h|\le 11$, shown on the left. The variable $h$ is on the horizontal axis.}
\label{fig:local_CRT_N=200_shift=0.2}
\end{figure}

\begin{figure}[h]
{\centerline{
{\hbox{
{\epsfig{file={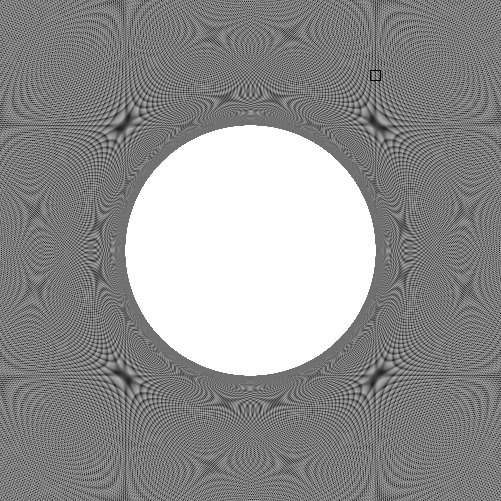}, height=4.5cm}}
%{\epsfig{file={Plot of recon_global_nal=400_eps=0.01.jpg}, height=3.5cm}}
}}}}
\caption{CRT reconstruction of the region $|x_1|,|x_2|\le 10$: $\e=0.01$, $N_{\al}=400$, $\de=0.03$.}
\label{fig:global_CRT_N=400}
\end{figure}

\begin{figure}[h]
{\centerline{
{\hbox{
{\epsfig{file={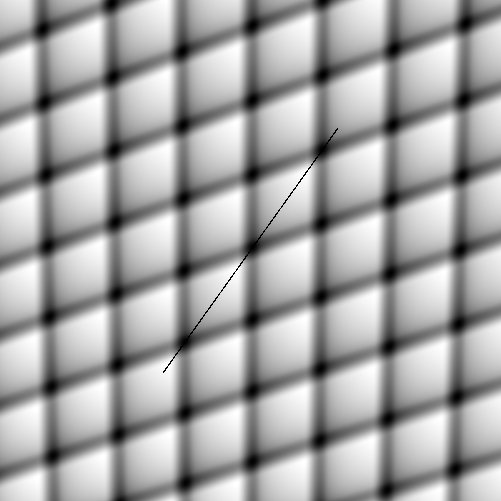}, height=4.5cm}}
{\epsfig{file={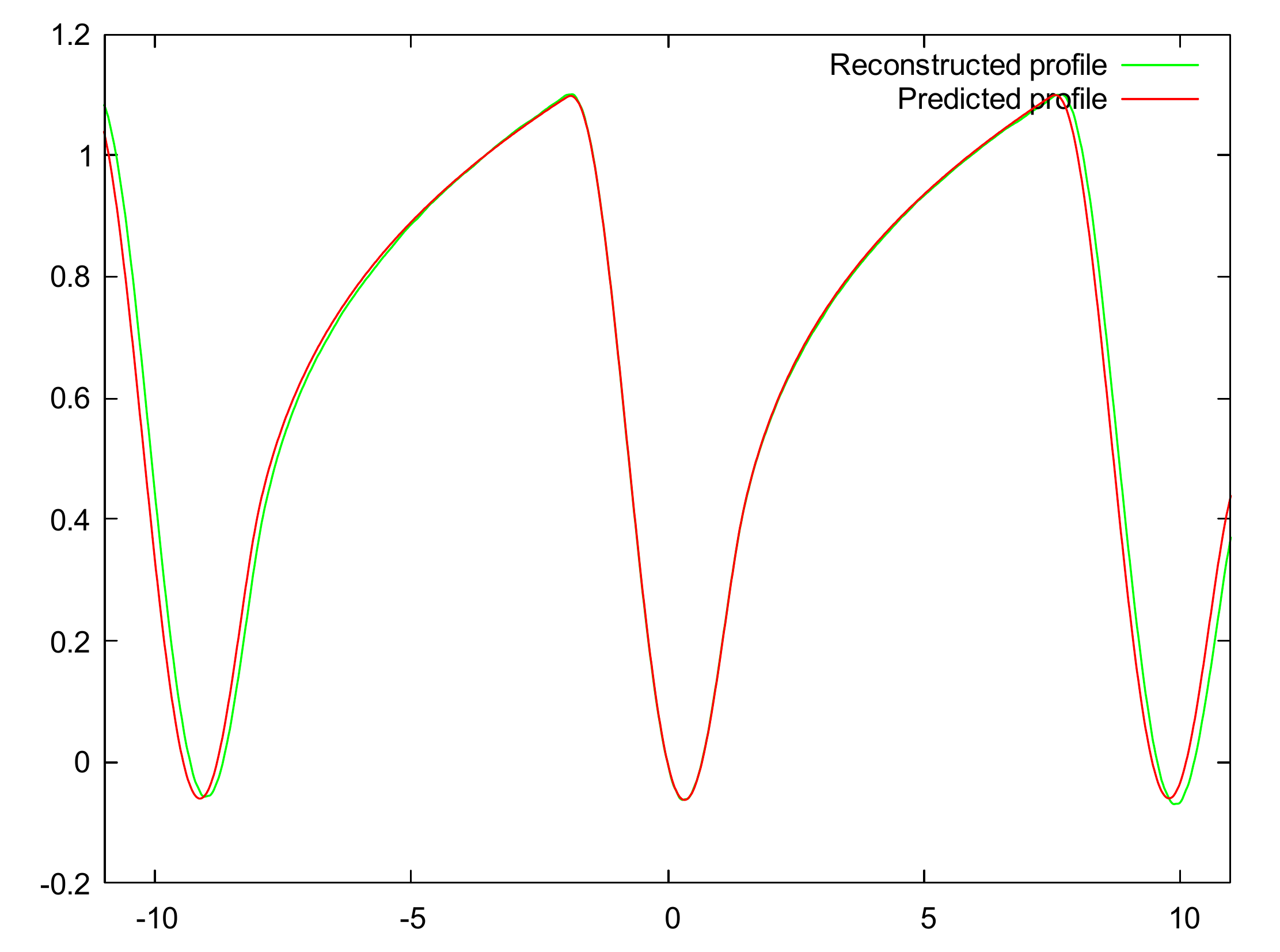}, height=4.5cm}}
}}}}
\caption{ROI CRT reconstruction: $\e=0.01$, $N_{\al}=400$, $\de=0.03$. The ROI is the square shown on the left in Figure~\ref{fig:global_CRT_N=400}. Left: reconstructed ROI, right: reconstructed (green) and predicted (red) profiles along the line segment $x=x_0+\e h\vec\Theta$, $|h|\le 11$, shown on the left. The variable $h$ is on the horizontal axis.}
\label{fig:local_CRT_N=400_shift=0.03}
\end{figure}

\begin{figure}[h]
{\centerline{
{\hbox{
{\epsfig{file={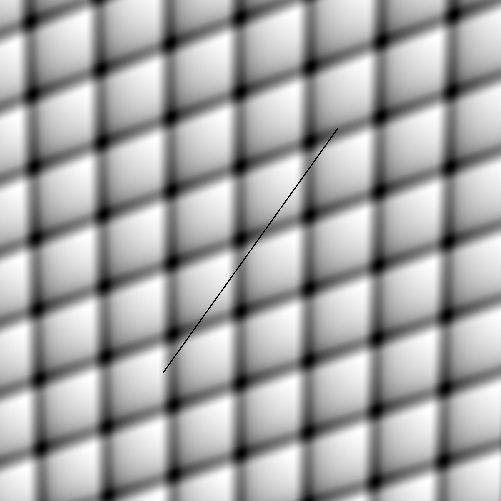}, height=4.5cm}}
{\epsfig{file={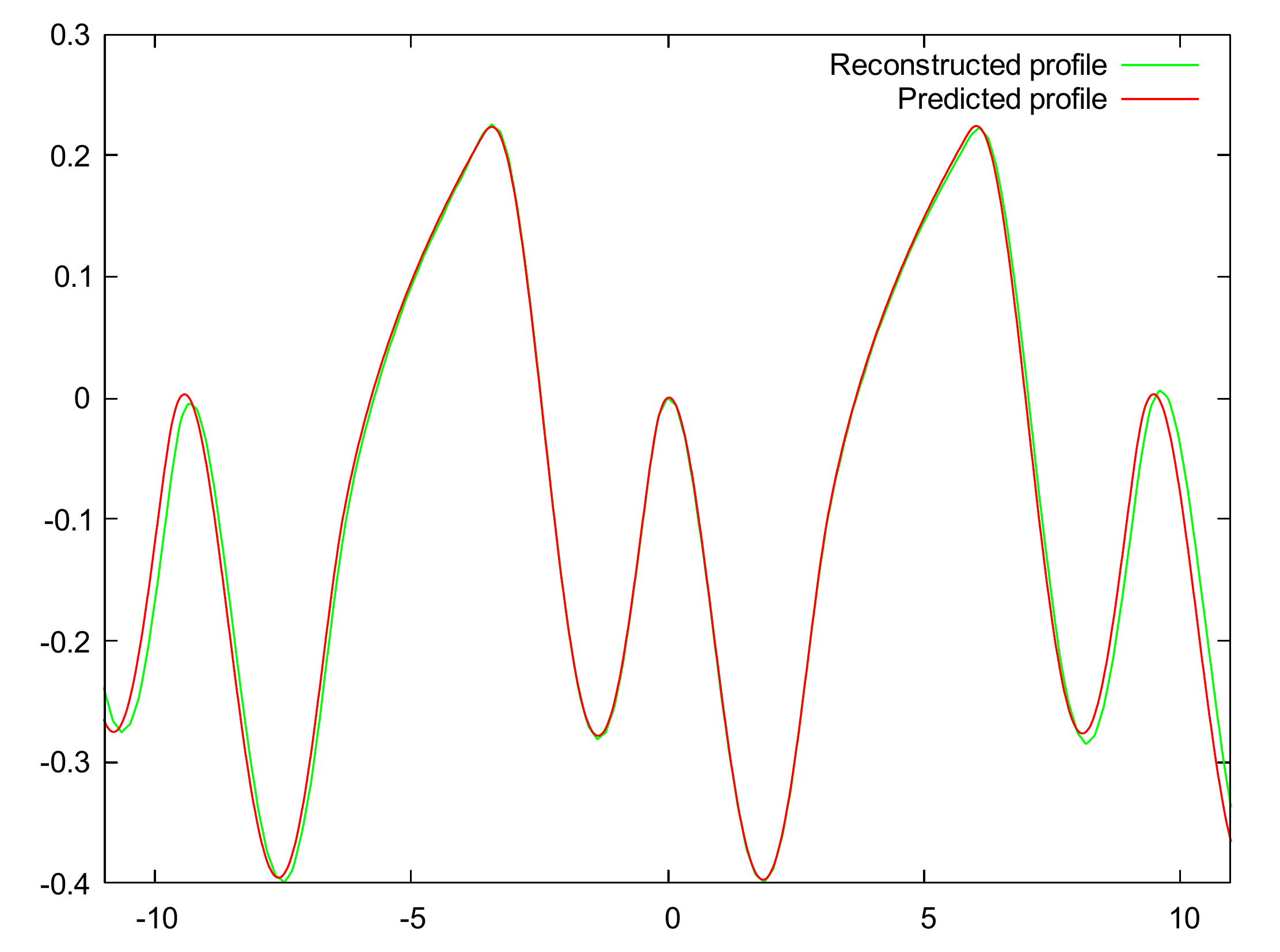}, height=4.5cm}}
}}}}
\caption{ROI CRT reconstruction: $\e=0.01$, $N_{\al}=400$, $\de=0.2$. The ROI is the square shown on the left in Figure~\ref{fig:global_CRT_N=400}. Left: reconstructed ROI, right: reconstructed (green) and predicted (red) profiles along the line segment $x=x_0+\e h\vec\Theta$, $|h|\le 11$, shown on the left. The variable $h$ is on the horizontal axis.}
\label{fig:local_CRT_N=400_shift=0.2}
\end{figure}

In the first experiment, $\e=0.02$, $N_\al=200$, and in the second: $\e=0.01$, $N_\al=400$. Since the direction $\alst =0$ is special, we use a non-zero shift $\de$ in \eqref{CRT data} for additional generality. The results are shown in Figures~\ref{fig:global_CRT_N=200} -- \ref{fig:local_CRT_N=400_shift=0.2}. 

Figure~\ref{fig:global_CRT_N=200} (left panel) shows the reconstructed region $|x_1|,|x_2|\le 10$ with $\e=0.02$ and $N_\al=200$. The left panel also shows the ROI (a small square). The right panel shows a line profile through the origin to confirm the accuracy of reconstruction. 
Figure~\ref{fig:local_CRT_N=200_shift=0.03} shows the reconstructed ROI with $\de=0.03$. The right panel shows the profiles of the reconstructed difference $\e^{-1/2}(\frec(x)-\frec(x_0))$ (green) and the prediction given by the main term on the right in \eqref{recon-2} (red) along the line segment $x=x_0+\e h\vec\Theta$, $|h|\le 11$, where $\vec\Theta=x_0/|x_0|$. The line segment is indicated on the left panel. The values of $h$ are on the horizontal axis of the profile. \bt{From \eqref{u0 mu0}, the values of $u_0\cdot\check x$ used in \eqref{recon-2} are given by $h\,\vec\alst\cdot\Theta$.}

Similarly, Figure~\ref{fig:local_CRT_N=200_shift=0.2} shows the reconstructed ROI and line profiles for the same line segment when $\de=0.2$. 

Figure~\ref{fig:global_CRT_N=400} shows the reconstructed region $|x_1|,|x_2|\le 10$ with $\e=0.01$ and $N_\al=400$. The ROI is indicated on the left panel. Recall that the size of the ROI is proportional to $\e$. Figure~\ref{fig:local_CRT_N=400_shift=0.03} shows the ROI and the corresponding line profiles for $\de=0.03$. Similarly, Figure~\ref{fig:local_CRT_N=400_shift=0.2} shows the reconstructed ROI and line profiles when $\de=0.2$. In both cases, the vector $\vec\Theta$ and the range of $h$ that determine the line segment are the same as before.

Comparing Figure~\ref{fig:local_CRT_N=200_shift=0.03} with Figure~\ref{fig:local_CRT_N=400_shift=0.03} and Figure~\ref{fig:local_CRT_N=200_shift=0.2} with Figure~\ref{fig:local_CRT_N=400_shift=0.2}, we see that reducing $\e$ and $\Delta \al$ improves the match between the reconstruction and prediction.

\subsection{Circular Radon transform}\label{sec:grt_exp}

In this subsection we experiment with the generalized Radon transform (GRT), which integrates over circles with any radius $\rho>0$ and centers on the circle $|x|=R$:
\be\label{GRT example}\begin{split}
(\R f)(\al,\rho)=&\hat f(\al,\rho)=\int_{S_{\al,\rho}}f(x)\dd x,\ \vec\al=(\cos\al,\sin\al),\\ 
\mps(\al,x):=&|x-R\vec\al|,\ S_{\al,\rho}:=\{x\in\br^2:|x-R\vec\al|=\rho\}.
\end{split}
\ee
The value of $R$ is fixed. Therefore
\be\label{grt aux}\begin{split}
&\dd_x\mps(\al,x)=\frac{x-R\vec\al}{|x-R\vec\al|},\ 
M=(1/\rhost)-(\vec\Theta_0^\perp\cdot\pa_y)^2 H(y)|_{y=y_0}>0.
\end{split}
\ee
\bt{In the computation of $M$ we used that $\dd_x\mps(\al,x)=1$.}
Reconstruction is achieved using a straightforward modification of \eqref{recon-orig} 
\be\label{recon-orig-GRT}\begin{split}
\frec(x)=&-\frac{\Delta\al}{2\pi}\sum_{\al_k\in\Omega} \frac{1}\pi \int \frac{\pa_\rho\hat f_\e(\al_k,\rho)}{p-\mps(x,\al_k)}\dd \rho,\\
\hat f(\al_k,\rho)=&\int w_\e(\rho-\rho’)\hat f(\al_k,\rho’)\dd\rho’,\ \al_k=(2\pi/N_{\al})k,\ w(\rho)=(15/16)(1-\rho^2)_+^2,
\end{split}
\ee
i.e. $w$ is the same as in \eqref{apert fn}. Clearly, the reconstruction is not theoretically exact anymore. But it preserves the strength of the singularities (in the Sobolev scale). Again, the weights in both the Radon transform and the inversion formula are set to 1: $W(\al,\rho;x)\equiv1$, $\omega(\al,x)\equiv1$.

The function $f$ is the characteristic function of the disk centered at $x_c$ with radius $r$. Thus, $\s=\{x\in\br^2:|x-x_c|=r\}$, \bt{see Figure~\ref{fig:GRT_diagr}}.

At a given $x_0$, aliasing arises due to the parts of $\s$ where various $\s_{\al,\rho}\ni x_0$ are tangent to $\s$. All such $(\al,\rho)$ are found by solving each of the two equations 
\be\label{eq rho v al}
|x_c-R\vec\al|-|x_0-R\vec\al|=\pm r
\ee
for $\al$ and setting $\rho=|x_0-R\vec\al|$. Generally, up to four solutions $(\al,\rho)$ (i.e., up to four circles $\s_{\al,\rho}$) can exist. To simplify the experiment, we reverse the argument. We pick some pair $(\alst ,\rhost)$ such that $\sst$ is tangent to $\s$ at some $y_0$, and then select some $x_0\in\sst$. To be specific, we select a ‘$+$’ in \eqref{eq rho v al}, i.e. $\rhost$ satisfies $|x_c-R\vec\alst|=r+\rhost$. This implies that $M=(1/r)+(1/\rhost)$, and $\vec\Theta_0=(y_0-R\vec\alst)/|y_0-R\vec\alst|$ points towards the center of curvature of $\s$ at $y_0$ \bt{(see Figure~\ref{fig:GRT_diagr}). Similarly to the classical Radon transform, our construction ensures that $\s_{\alst,\rhost+\de}$, where $0<\de\ll1$, intersects $\s$ at two points.}

\begin{figure}[h]
{\centerline{
{\epsfig{file={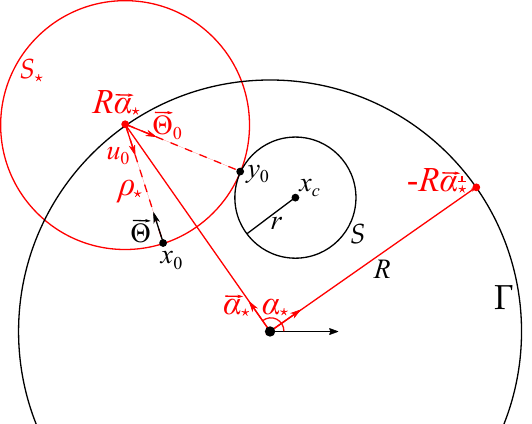}, width=8cm}}
}}
\caption{\bt{Illustration of various quantities used in the main formula \eqref{recon-2} to predict aliasing from a disk in the case of the circular Radon transform.}}
\label{fig:GRT_diagr}
\end{figure}

To illustrate aliasing only from the place where $\sst$ is tangent to $\s$ we select $\Omega$ to be a sufficiently small neighborhood of $\alst $.  
\bt{Since $\mps(\al,x)=|x-R\vec\al|$ and $\mP(\al)=|x_c-R\vec\al|-r$,} we find
\be\label{u0mu0 grt}\begin{split}
u_0=&\frac{x_0-R\vec\al_\star}{|x_0-R\vec\al_\star|},\\
\mu_0=&
%&-R\vec\al_\star^\perp\cdot\left(u_0-\frac{x_c-R\vec\al_\star}{|x_c-R\vec\al_\star|}\right)=
-R\vec\al_\star^\perp\cdot(u_0-\vec\Theta_0)=-(R/\rhost)\vec\al_\star\cdot(x_0-y_0),
\end{split}
\ee
\bt{see Figure~\ref{fig:GRT_diagr}.}

For reconstructions we use 
\be\label{GRT param values}\begin{split}
&R=5,\ x_c=(1,1),\ r=2,\ (\alst ,\rhost)=(0.53\pi,2.24),\ 
x_0=(-1.42,2.95),\\ &\Omega:=[\alst -\pi/4,\alst +\pi/4].
\end{split}
\ee
In the first reconstruction, $\e=10^{-2}$, $N_\al=500$, and in the second: $\e=0.5\cdot 10^{-2}$, $N_\al=1000$. The results are shown in Figures~\ref{fig:GRT_N=500} and \ref{fig:GRT_N=1000}, respectively. The left panels show the limited angle reconstruction of the region $|x_1|,|x_2|\le 4$. The middle panels show the limited angle reconstruction of an ROI. The ROI is a small square centered at $x_0$ with side length $40\e$, the ROI is shown on the left panel. The right panels show the profiles of the reconstructed difference $\e^{-1/2}(\frec(x)-\frec(x_0))$ (green) and the prediction given by the main term on the right in \eqref{recon-2} (red) along the line segment $x=x_0+\e h\vec\Theta$, $|h|\le 6$, shown in the middle panel. The values of $h$ are on the horizontal axis of the profiles. The unit vector $\vec\Theta$ is chosen to be orthogonal to $\sst$ at $x_0$ (i.e., $\vec\Theta$ and $u_0$ are parallel, \bt{see Figure~\ref{fig:GRT_diagr}}). In the experiments we set $\vec\Theta=-u_0$. As is seen, reducing $\e$ and $\Delta \al$ improves the match between the reconstruction and prediction.

\begin{figure}[h]
{\centerline{
{\hbox{
{\epsfig{file={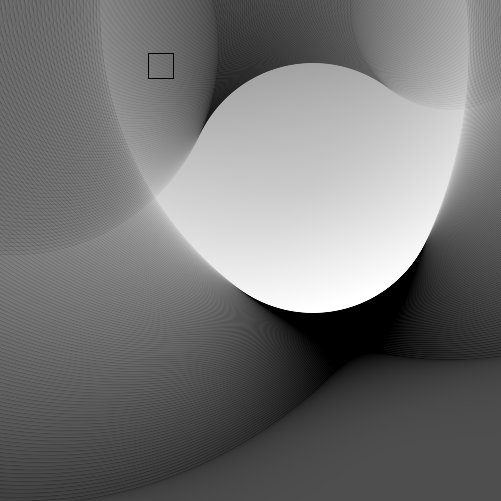}, height=3.5cm}}
{\epsfig{file={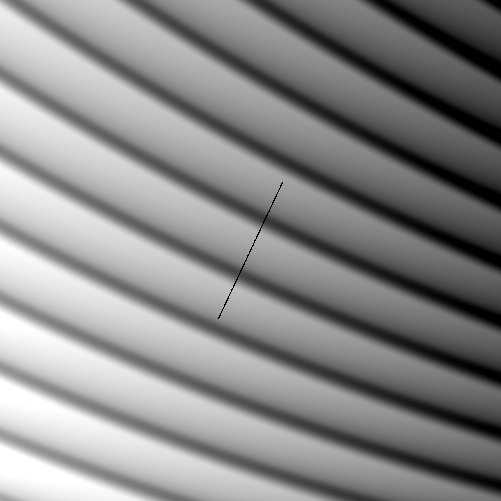}, height=3.5cm}}
{\epsfig{file={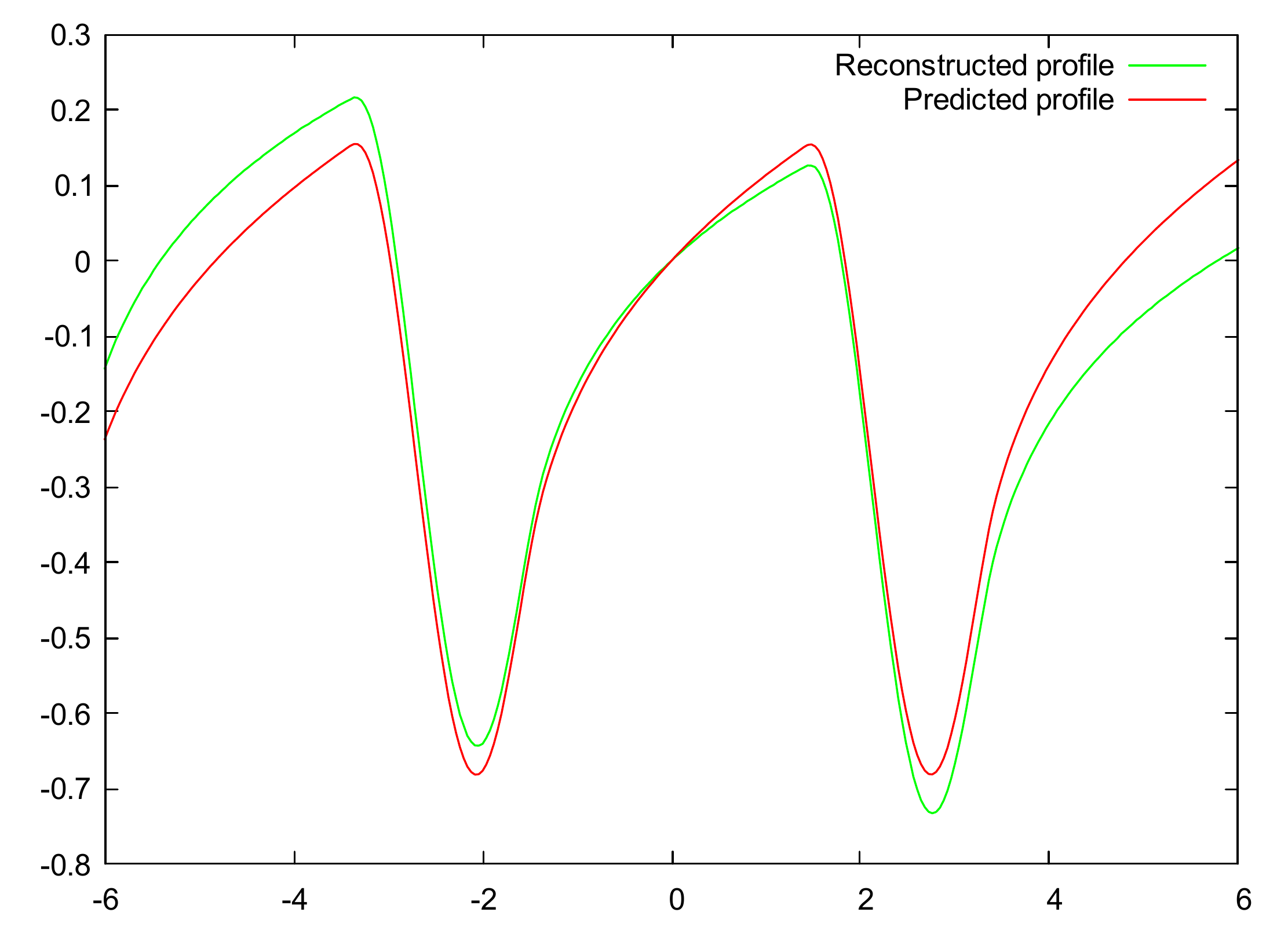}, height=3.5cm}}
}}}}
\caption{Limited angle GRT reconstruction: $\e=0.01$, $N_{\al}=500$. Left: global reconstruction, middle: reconstruction inside the square ROI shown on the left, right: profiles of the reconstruction (green) and prediction (red) along the line segment $x=x_0+\e h\vec\Theta$, $|h|\le 6$, shown in the middle. The variable $h$ is on the horizontal axis.}
\label{fig:GRT_N=500}
\end{figure}

\begin{figure}[h]
{\centerline{
{\hbox{
{\epsfig{file={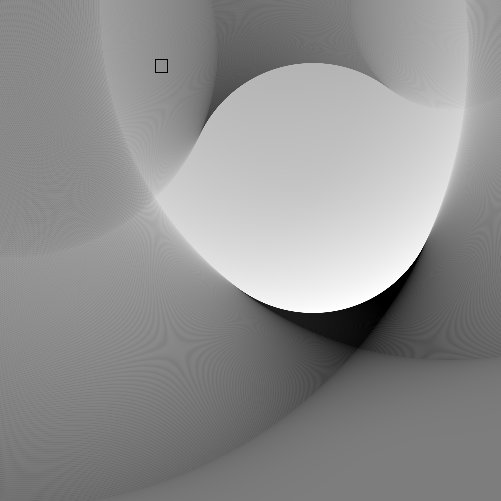}, height=3.5cm}}
{\epsfig{file={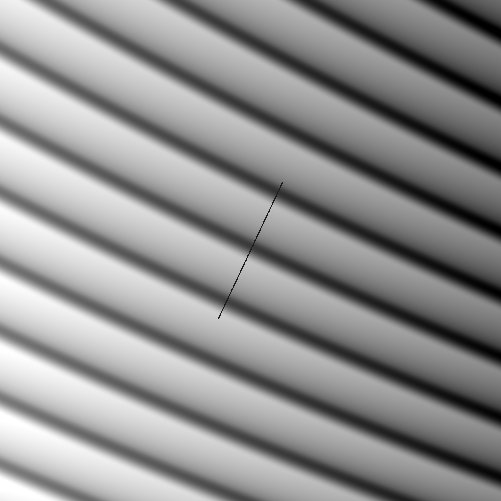}, height=3.5cm}}
{\epsfig{file={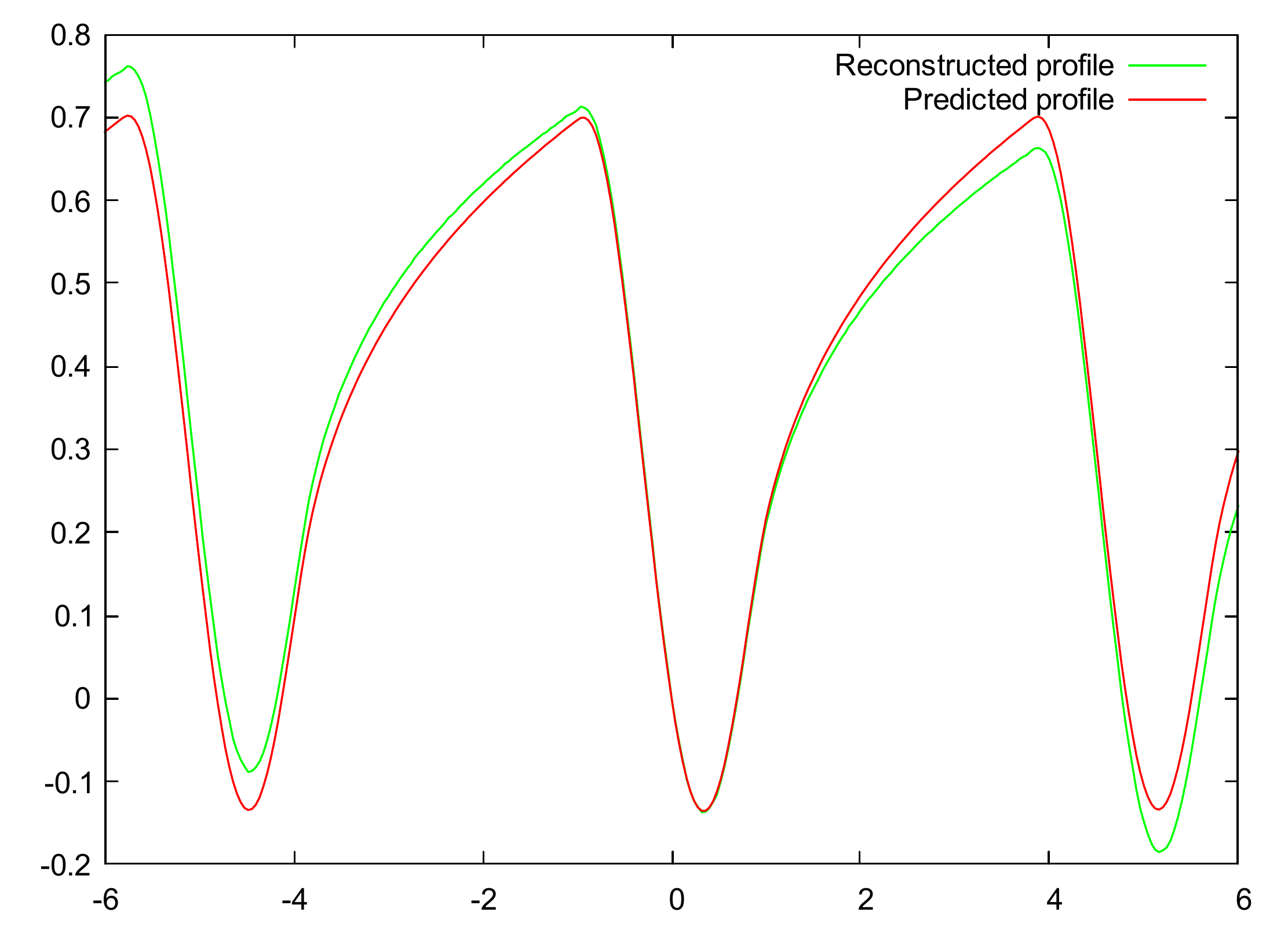}, height=3.5cm}}
}}}}
\caption{Limited angle GRT reconstruction: $\e=0.005$, $N_{\al}=1000$. Left: global reconstruction, middle: reconstruction inside the square ROI shown on the left, right: profiles of the reconstruction (green) and prediction (red) along the line segment $x=x_0+\e h\vec\Theta$, $|h|\le 6$, shown in the middle. The variable $h$ is on the horizontal axis.}
\label{fig:GRT_N=1000}
\end{figure}

\appendix

\section{Proofs of lemmas}\label{sec:proofs of lemmas}

\subsection{Proof of Lemma~\ref{lem:two Ps}}\label{sec:Psm}
The property $\mP(\alst )=\pst$ follows from assumption~\ref{ass:Phi}(2). Recall that $H(y)=0$ is a local equation of $\s$ 
%and \textcolor{red}{$\dd H(y_0)=\vec\Theta_0$} 
(cf. \eqref{Phi eqs} and the paragraph preceding it). To find $\mP(\al)$, we solve
\be\label{eq for P}
H(y)=0,\ \la\dd H(y)=\dd_y\mps(\al,y)
\ee
for $y\in\s$ and \bt{$\la$} in terms of $\al$ near $(\bt{\la=1,}y=y_0,\al=\alst)$ and then set $\mP(\al)=\mps(\al,y(\al))$. Assumptions~\ref{ass:Phi}(1, 2, 4) and the Implicit Function Theorem imply that $y(\al)$ and, therefore, $\mP(\al)$ are smooth \bt{in a small neighborhood $\Omega\ni\alst$.} Since $y’(\al)$ is tangent to $\s$, using the second equation in \eqref{eq for P} gives $\mP’(\alst)=\pa_\al\mps(\alst,y_0)$.

\subsection{Proof of Lemma~\ref{lem:first res}}\label{ssec:main term}

Denote
\be\label{H def}
H(x,\al,\e):=\frac{\mps(\al,x)-\mP(\al)}\e,\ x=x_0+\e\check x,\al\in\Omega.
\ee
Since $\mu_0\not=0$ (cf. \eqref{two derivs 1}), we have $|\mps(\al,x_0)-\mP(\al)|\ge c|\al-\alst |$ for any $\al\in\Omega$ and some $c>0$. Hence 
\be\label{H ineq}
|H(x,\al_k,\e)|\ge c_1|k-\kst|,  \text{ for all } |\check x|\le c,\al_k\in\Omega,\bt{|k-\kst|\ge c_2,}
\ee
for some $c,c_1,c_2>0$, and all $\e>0$ sufficiently small. From \eqref{recon-1},
\be\label{recon-2 aux}\begin{split}
&\freco(x_0+\e\check x)-\freco(x_0)=-\frac{\Delta\al}{2\pi\e^{1/2}}\bigl(J+O(\e^{1/2})\bigr),\\
&J:=\sum_{\al_k\in\Omega} \omega(\al_k,x_0)\varphi_1(\al_k)
\left[\psi\left(H(x_0+\e\check x,\al_k,\e)\right)-\psi\left(H(x_0,\al_k,\e)\right)\right].
\end{split}
\ee
The $O(\e^{1/2})$ term in parentheses on the right in \eqref{recon-2 aux} denotes the contribution, which arises due to the $x$-dependence of $\om$ in \eqref{recon-1}. Here we use \eqref{psi props} with $n=0$, \eqref{H ineq}, and that for some $c$ and all $\check x$ in a bounded set:
\be
|\om(\al,x_0+\e\check x)-\om(\al,x_0)|\le c\e,\ |\varphi_1(\al)|\le c,\ \al\in\Omega,
\ee
hence\bt{
\be
O(\e)\biggl(O(1)+\sum_{c_2\le|k-\kst|\le O(1/\e)}|k-\kst|^{-1/2}\biggr)=O(\e^{1/2}).
\ee}

%Based on \eqref{recon-2 aux} we consider
%\be\label{main expr}
%J:=\sum_{\al_k\in\Omega} \omega(x_0,\al_k)\varphi_1(\al_k)\left[\psi(H(x_0+\e\check x,\al_k,\e))-\psi(H(x_0,\al_k,\e))\right].
%\ee
%where $\varphi(\al):=\omega(x_0,\al)\varphi_1(\al)$. 
From \eqref{two derivs 1},
\be\label{psi arg x}\begin{split}
H(x_0+\e\check x,\al,\e)&=\bt{H(x_0,\al,\e)+\dd_x \mps(\al,x_0) \check x+O(\e)}\\
&=H(x_0,\al,\e)+u_0\cdot \check x+O(\e+|\al-\alst |).
\end{split}
\ee
 Also, $|\omega(x_0,\al)\varphi_1(\al)|\le c$ for some $c$ and all $\al\in\Omega$. Therefore, by \eqref{psi props} with $n=1$ and \eqref{H ineq},
\be\label{del del psi}\begin{split}
\sum_{\al_k\in\Omega}&\omega(\al_k,x_0)\varphi_1(\al_k)\bigl[\psi(H(x_0+\e\check x,\al_k,\e))-\psi(H(x_0,\al_k,\e)+u_0\cdot \check x)\bigr]\\
&=\sum_{|k|\le O(1/\e)}\psi’(H(x_0,\al_k,\e)+O(1))O(\e+\e|k-\kst |)\\
&=O(\e)\biggl(1+\sum_{\bt{c_2\le}|k-\kst|\le O(1/\e)}\frac{1+|k-\kst |}{|k-\kst |^{3/2}}\biggr)=O(\e^{1/2}).
\end{split}
\ee
\bt{Here we use that $w’\in L^q(\br)$, $q>2$ (see assumption~\ref{ass:w}(1)), so $\psi’$ is continuous. 
This shows that if $w$ does not have the required smoothness (e.g., if $w$ is the characteristic function of a detector pixel), the magnitude of the expression in \eqref{del del psi} may turn out to be much larger, leading to a slower rate of convergence in Theorem~\ref{main res} (or even to a breakdown of the convergence altogether).}

From \eqref{recon-2 aux}, \eqref{psi arg x}, and \eqref{del del psi},
\be\label{main expr v2}\begin{split}
&J=\sum_{\al_k\in\Omega} \omega(\al_k,x_0)\varphi_1(\al_k)\Delta\psi(H(x_0,\al_k,\e))+O(\e^{1/2}),\\ &\Delta\psi(t):=\psi(t+u_0\cdot \check x)-\psi(t).
\end{split}
\ee
%Recall that $\Delta\psi$ is defined in \eqref{del psi}.
Furthermore,
\be\label{psi arg 1}
H(x_0,\al_k,\e)=\mu_0\frac{\al_k-\alst }\e+R_k,\ 
R_k=O(\e(k-\kst )^2).
\ee
Denote, for simplicity, $a_k=\mu_0\kappa(k-\kst )$. Then
\be\label{psi simpl 2}
\Delta\psi(a_k+R_k)-\Delta\psi(a_k)=R_k\Delta\psi’(a_k+\xi_k),
\ee
where $|\xi_k|\le |R_k|$.
We can assume that $\Omega$ is sufficiently small, so that
\be
|\mu_0(\al_k-\alst )+\e R_k|\ge c|\al_k-\alst |,\ \forall \al_k\in\Omega,
\ee
for some $c>0$. Dividing by $\e$ implies  
\be\label{arg est}
\left|a_k+R_k\right|\ge c\kappa|k-\kst |,\ \forall\al_k\in\Omega,
\ee
with the same $c$. Using \eqref{psi props} with $n=2$ gives
\be\label{another del}\begin{split}
\sum_{\bt{c\le}|k-\kst|\le O(1/\e)}&\omega(\al_k,x_0)\varphi_1(\al_k)\bigl[\Delta\psi(a_k+R_k)-\Delta\psi(a_k)\bigr]\\
&=\sum_{\bt{c\le} |k-\kst|\le O(1/\e)}\frac{O(\e|k-\kst |^2)}{1+|k-\kst |^{5/2}}=O(\e^{1/2}),
\end{split}
\ee
\bt{for some $c>0$ sufficiently large. The requirement $|k-\kst|\ge c$ is needed, because $\psi^{\prime\prime}(\hat q)$, on which the estimate \eqref{another del} is based, may not exist for $\hat q$ in a compact set when $w’\in L^q$. To estimate the remaining finitely many terms without appealing to the second derivative we write
\be\label{fin many}\begin{split}
&\Delta\psi(a_k+R_k)-\Delta\psi(a_k)\\
&=\bigl[\psi(a_k+u_0\cdot \check x+R_k)-\psi(a_k+u_0\cdot \check x)\bigr]-\bigl[\psi(a_k+R_k)-\psi(a_k)\bigr]\\
&=O(\e),\ |k-\kst|\le c.
\end{split}
\ee
This follows, because $\psi’$ is continuous on all of $\br$, and $R_k=O(\e)$ whenever $|k-\kst|\le c$ (cf. \eqref{psi arg 1}). This is another place where we use that $w’\in L^q$. If $w$ is not sufficiently smooth, the quantity in \eqref{fin many} may turn out to be much larger.}

It is clear that all the big-$O$ terms are uniform with respect to $\check x$ (and, hence, $h$) restricted to a bounded set. Combining \eqref{recon-2 aux}, \eqref{main expr v2}, \eqref{another del}, \bt{and \eqref{fin many}} finishes the proof.

\subsection{Proof of Lemma~\ref{lem:diff phi3}}\label{sec:ramp phi}

Denote
\be\label{diff alt}
J:=\int  s^{-1} \left[\varphi_3’(s+q+\Delta q)-\varphi_3’(s+q)\right] \dd s,
\ee
where we omitted the dependence on $\al$ for simplicity. All the big-$O$ terms in this subsection are uniform with respect to $\al\in\Omega$. 
Restricting the integral in \eqref{diff alt} to $|s|\le 1$ we find
\be\label{diff alt 1}
J_1:=\int_{|s|\le 1}  s^{-1} \left(\left[\varphi_3’(s+q+\Delta q)-\varphi_3’(q+\Delta q)\right]-\left[\varphi_3’(s+q)-\varphi_3’(q)\right]\right) \dd s. 
\ee
Clearly, $J_1=O(|\Delta q|)$ uniformly in $|q|\le c$. Here we have used that $\varphi_3$ is smooth, so its third order derivative is bounded on compact sets. By \eqref{f3 props}, $\varphi_3^{\prime\prime}(p) =O(|p|^{-3/2}),\ p\to\infty$. Hence
\be\label{diff J2}
J_2:=\int_{|s|\ge 1}  s^{-1} \left[\varphi_3’(s+q+\Delta q)-\varphi_3’(s+q)\right] \dd s=O(|\Delta q|)
\ee
uniformly in $|q|\le c$. Combining the estimates for $J_{1,2}$ proves the lemma.

\subsection{Proof of Lemma~\ref{lem:aux conv}}\label{sec:aux conv} \color{black}
The Euler-MacLauren formula reads as follows \cite[eq. (25.7)]{Kac2002}:
\be\label{EML}\begin{split}
\sum_{k=a}^{b-1} f(k)=&\int_a^b f(t)\dd t+\sum_{m=1}^{N’}\frac{b_m}{m!}(f^{(m-1)}(b)-f^{(m-1)}(a))\\
&-\int_a^b \frac{B_{N’}(\{1-t\})}{N’!} f^{(N’)}(t)\dd t.
\end{split}
\ee
Here $b>a$ are integers, $B_m$ and $b_m$ are Bernoulli polynomials and numbers, respectively, $\{t\}=t-\lfloor t\rfloor$ is the fractional part of $t\in\br$, and $\lfloor t\rfloor$ is the floor function, i.e. the largest integer not exceeding $t$. 

Substituting $f(t)=g(\e t)$, taking the limit as $a\to-\infty$, $b\to\infty$ (which is allowed due to the decay of $g$ and its derivatives), changing variables $\tau=\e t$, and using that $g^{(N’)}\in L^1(\br)$, we finish the proof. \color{black}

\bibliographystyle{ieeetr}
\bibliography{My_Collection}
\end{document}